\documentclass[11pt]{amsart}
\usepackage[margin=1.2in]{geometry}

\usepackage{mathrsfs}
\usepackage{amsfonts}
\usepackage{latexsym,epsfig}
\usepackage{mathabx, stmaryrd}
\usepackage{amsmath, amsthm, amscd, amssymb, centernot}
\usepackage[pdftex]{color}
\usepackage{comment}
\usepackage{marginnote}
\usepackage[colorlinks,citecolor=cyan,linkcolor=magenta]{hyperref}
\usepackage[nameinlink]{cleveref}
\usepackage{enumitem}
\usepackage[normalem]{ulem}

\usepackage{tikz}
\usepackage{tikz-cd}
\usetikzlibrary{matrix,decorations.pathmorphing,arrows}
\tikzset{commutative diagrams/.cd,arrow style=tikz,diagrams={>=stealth'}}
\usepackage{todonotes}

\newcommand{\RR}[0]{\mathbb{R}}

\newcommand{\Z}{\mathbb{Z}}

\newcommand{\R}{\mathbb{R}}

\newcommand{\LL}{{\mathcal L}}
\newcommand{\wt}{\widetilde}
\newcommand{\mc}{\mathcal}

\newcommand{\ol}{\overline}

\newcommand{\del}{\partial}

\DeclareMathOperator{\fr}{fr}

\DeclareMathOperator{\Fix}{Fix}
\DeclareMathOperator{\stab}{stab}

\newcommand{\FF}{\mathcal{F}}

\newcommand{\orb}{\mathcal{O}}

\newcommand\tsim{\kern-.4em\sim}

\newcommand\ssm{\smallsetminus}

\renewcommand{\phi}{\varphi}
\renewcommand{\epsilon}{\varepsilon}

\DeclareMathOperator{\intr}{int}

\DeclareMathOperator{\Homeo}{Homeo}
\DeclareMathOperator{\Stab}{stab}

\newcommand{\univ}{S^1_{\mathrm{univ}}}

\usepackage{hyperref}

\usepackage{thmtools}

\numberwithin{equation}{section}

\newtheorem{theorem}[equation]{Theorem}
\newtheorem{lemma}[equation]{Lemma}
\newtheorem{proposition}[equation]{Proposition}
\newtheorem{corollary}[equation]{Corollary}
\newtheorem{conjecture}[equation]{Conjecture}

\newtheorem{claim}[equation]{Claim}

\newtheorem{observation}[equation]{Observation}
\theoremstyle{definition}
\newtheorem{question}[equation]{Question}

\declaretheorem[style=definition,qed=$\lozenge$,sibling=theorem]{definition}
\declaretheorem[style=definition,qed=$\lozenge$,sibling=theorem, name=Construction]{construction}

\begin{document}

\title[]{Simultaneous universal circles}
\author[M.P. Landry]{Michael P. Landry}
\address{Department of Mathematics\\
Saint Louis University}
\email{\href{mailto:michael.landry@slu.edu}{michael.landry@slu.edu}}
\author[Y.N. Minsky]{Yair N. Minsky}
\address{Department of Mathematics\\ 
Yale University}
\email{\href{mailto:yair.minsky@yale.edu}{yair.minsky@yale.edu}}
\author[S.J. Taylor]{Samuel J. Taylor}
\address{Department of Mathematics\\ 
Temple University}
\email{\href{mailto:samuel.taylor@temple.edu}{samuel.taylor@temple.edu}}
\date{\today}
\thanks{Landry was partially supported by NSF grant DMS-2405453.
Minsky was partially supported by DMS-2005328.
Taylor was partially supported by DMS-2102018 and the Simons Foundation.}

\begin{abstract}
Let $\phi$ be a pseudo-Anosov flow on a closed oriented atoroidal $3$-manifold $M$. We show that if $\mc F$ is any taut foliation almost transverse to $\phi$, then the action of $\pi_1(M)$ on the boundary of the flow space, together with a natural collection of explicitly described monotone maps, defines a universal circle for $\mc F$ in the sense of Thurston and Calegari--Dunfield. 
\end{abstract}
\maketitle

\setcounter{tocdepth}{1}
%\tableofcontents

\section{Introduction}

Starting with Thurston's seminal work in the late 1990s \cite{thurston1997three, thurston_circles2}, various notions of \emph{universal circle} have been formulated with the goal of associating a natural circle action to a geometric, topological, or dynamical structure on a three-dimensional manifold \cite{CalDun_UC, calegari2006universal, fenley2012ideal, schleimer2019veering, calegari2024zippers}. Such universal circles have been used to great effect by a number of authors (e.g. \cite{ Calegari2000Rcovered, Fenley2002Rcovered, frankel_closedorbits, 
BFM, boyer2023recalibrating}), 
and it is now a fundamental problem 
to relate these circle actions in order to better understand interactions among their underlying structures.

In this article we establish the connection between the two most well-known universal circles, namely those associated to taut foliations and to pseudo-Anosov flows. 
We describe some of the main objects before summarizing our results.
All of the 3-manifolds we consider are closed, connected, and oriented.

\subsection{Circle actions from taut foliations}

Given a taut foliation $\FF$ of an irreducible, atoroidal 3-manifold $M$, a
\emph{universal circle} for $\FF$ is a way of organizing the ideal geometry of the leaves
of $\mc F$. It consists of a faithful representation 
\[
\rho\colon \pi_1(M)\to \Homeo^+(\univ)
\]
where $\univ$ is a circle, 
together with a family of monotone maps.
For each leaf of the lift of $\FF$ to the universal cover of $M$, there is an associated monotone map from $\univ$ to the ideal boundary of the leaf. The collection of these ideal boundaries is acted upon by $\pi_1(M)$, and the monotone maps are in particular required to intertwine this action with $\rho$. For more, see \Cref{UCdef}.

Thurston and Calegari--Dunfield proved that every taut foliation of a closed, oriented, irreducible, atoroidal 3-manifold has a universal circle \cite{CalDun_UC}. We remark that the word ``universal" does not refer to any kind of uniqueness: a universal circle for such a foliation $\mc F$ is in fact never unique, although each universal circle does have a unique minimal quotient in a certain sense (see \cite[\S 5.1]{Calegari_promoting}).

\subsection{Circle actions from pseudo-Anosov flows}

Let $\phi$ be a pseudo-Anosov flow on a 3-manifold $M$ (see \Cref{sec:pAflows}), and let $\wt \phi$ be its lift to the universal cover $\wt M$ of $M$. Fenley and Mosher proved that the quotient of $\wt M$ by the flowlines of $\wt\phi$ is a plane $\orb$, on which $\pi_1(M)$ acts by orientation preserving homeomorphisms. The stable and unstable foliations of $\phi$ induce singular foliations on $\orb$, and Fenley used these foliations to construct a natural compactification of $\orb$ by a circle $\del \orb$, such that the $\pi_1$-action on $\orb$ induces a faithful action on $\del\orb$ by orientation preserving homeomorphisms. See \Cref{sec:flowspace}.

\subsection{Summary of main results}

This article is concerned with the following situation: $\phi$ is a pseudo-Anosov flow on an an irreducible atoroidal 3-manifold $M$, and $\FF$ is a taut foliation of $M$ almost transverse to $\phi$. This means that $\FF$ is cooriented and positively transverse to a ``dynamic blowup" of $\phi$, see \Cref{sec:pAflows}. We will also refer to this blowup as $\phi$.

\begin{figure}
\centering
\includegraphics[height=2.5in]{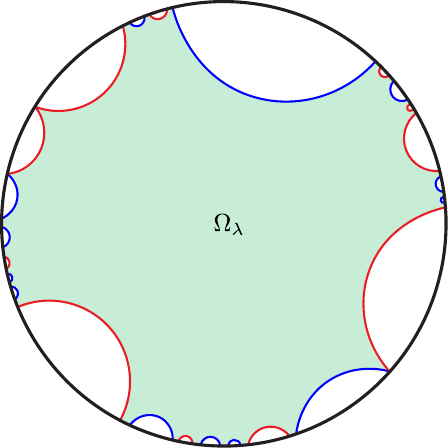}
\caption{For any leaf $\lambda$ of $\wt \FF$, its projection $\Omega_\lambda$ to the flow space $\orb$ of $\phi$ is a region bounded by lines in leaves of the stable and unstable foliations. These frontier components can be proper subsets of leaves, and can share ideal points in the ideal boundary of $\orb$. Although this picture may evoke geodesics in the hyperbolic plane, there is no natural geometric structure on $\orb$.}
\label{fig:shadow}
\end{figure}

Denote the lift of $\FF$ to $\wt M$ by $\wt \FF$. Each leaf $\lambda$ of $\wt\FF$ can be projected to the plane $\orb$, and Fenley showed that its projection $\Omega_\lambda$ is a region whose frontier is a disjoint union of properly embedded lines in leaves of the stable and unstable foliations \cite{Fen09}. The accumulation points of $\Omega_\lambda$ in $\del \orb$ form an uncountable closed set whose complement is a collection of open intervals spanned by the frontier components. See \Cref{fig:shadow}.

In \Cref{sec:monotonemaps}, we show that there is a natural monotone map $\pi_\lambda$ from $\del\orb$ to the ideal boundary of $\lambda$ that in particular collapses all these open intervals. We take this as our monotone map $\pi_\lambda\colon \partial \mc O \to \partial \lambda$ for each leaf $\lambda$ of $\wt \FF$. After verifying the other conditions of the definition, we prove this gives a universal circle for $\FF$ in \Cref{simUC}. Consequently:

\begin{theorem}[Universal circles from flow space boundaries]\label{maintheorem}
Let $M$ be an atoroidal $3$-manifold with a pseudo-Anosov flow $\varphi$, and let $\FF$ be any taut foliation almost transverse to $\phi$.
Then the action $\pi_1(M) \curvearrowright \del\orb$, together with the family of monotone maps $\{\pi_\lambda\}$, defines a universal circle for $\FF$. 
\end{theorem}
Further, we prove in \Cref{prop:when_min} that this is a ``minimal" universal circle unless $\phi$ is a skew Anosov flow which is not regulating for $\FF$; see \Cref{sec:minimal} for definitions. 

\smallskip

As a consequence, if $\FF$ is almost transverse to 2 distinct (not orbit equivalent) pseudo-Anosov flows, \Cref{maintheorem}, together with the main result of \cite{BFM}, gives 2 distinct (not conjugate) minimal universal circles for $\FF$. It is simple to build examples where this occurs using the Gabai--Mosher construction of pseudo-Anosov flows almost transverse to finite depth foliations. While this construction is not fully in the literature, it was partially written down by Mosher in \cite{Mos96} and is being finished by Landry--Tsang in \cite{LandryTsangStep1} and work in progress.

\smallskip
The proof of \Cref{maintheorem} is relatively nonconstructive, in the sense that it does not give an explicit description of the collapsing maps $\{\pi_\lambda\}$. While it tells us that the intervals in $\partial \orb$ spanned by frontier components of $\Omega_\lambda$ are collapsed by $\pi_\lambda$, it does not say whether any other collapsing happens. Via a detailed analysis of the flow space of $\phi$, we are able to give a clean characterization of each gap (i.e. interior of a nontrivial point preimage): 

\begin{theorem}[Characterizing the gaps] \label{th:gaps}
Each gap of $\pi_\lambda$ is spanned by a finite sequence of frontier components of $\Omega_\lambda$, each sharing an ideal endpoint with the next. Further, the length of such a sequence is uniformly bounded independently of $\lambda$. 
\end{theorem}
Translating the first sentence of this characterization to \Cref{fig:shadow}, it tells us that two points are identified by the monotone map $\pi_\lambda \colon \partial \orb \to \partial \lambda$ if and only if there is an arc in $\partial \orb$ between them meeting at most finitely many points in the limit set of $\Omega_\lambda$.

\subsection{The flows (almost) transverse to a foliation}

We conclude the introduction by discussing the following question. 

\begin{question}\label{qclassify}
Let $\FF$ be a taut foliation of an atoroidal 3-manifold $M$. How can we classify the collection of pseudo-Anosov flows almost transverse to $\FF$? 
\end{question}
Here, flows are considered up to orbit equivalence, by which we mean two flows are equivalent if there is a homeomorphism $M\to M$, isotopic to the identity, which sends flowlines to flowlines in an orientation-preserving manner.

\Cref{qclassify} has been completely answered if $\FF$ is $\R$-covered, meaning that the leaf space of its lift to the universal cover of $M$ is a line. 
Calegari and Fenley separately showed in \cite{Calegari2000Rcovered} and \cite{Fenley2002Rcovered} that in this case $\FF$ it is transverse to a regulating (see \Cref{sec:minimal}) pseudo-Anosov flow, which Fenley additionally showed in \cite[Theorem G]{fenley2012ideal} has no perfect fits (see \Cref{sec:flowspace}).
Fenley further showed in \cite{Fenley2013rigidity} that this flow is the unique pseudo-Anosov flow transverse to $\FF$, unless $\FF$ is topologically equivalent to the stable or unstable foliation of a skew Anosov flow $\phi_{\text{skew}}$, up to collapsing pockets of parallel leaves.
In the latter case, $\FF$ is transverse to exactly one additional flow: a new flow obtained by ``tilting" $\phi_{\text{skew}}$ slightly, which is equivalent to $\phi_{\text{skew}}$ by structural stability (see \cite{fenley2005regulating}). This tilted flow is not regulating for $\FF$. Finally, we remark that in the $\R$-covered case transversality and almost transversality are equivalent, so this resolves \Cref{qclassify} if $\FF$ is $\R$-covered.

Perhaps the simplest class of foliations which are not $\R$-covered is the \emph{depth one foliations}, i.e. foliations with finitely many compact leaves and all other leaves accumulating only on the compact ones.
Junzhi Huang has shown recently that if $\FF$ is a depth one foliation and $\phi$ is a pseudo-Anosov flow with no perfect fits transverse to $\FF$, the ``universal circle of leftmost sections" constructed by Calegari--Dunfield can be identified equivariantly with $\del \orb$ \cite{Huang_UC}. In this setting, Huang further shows that
 $\phi$ is the unique pseudo-Anosov flow with no perfect fits transverse to $\FF$.
Using \Cref{th:gaps} and Huang's result, we prove in \Cref{thm:depth1rigidity} that in fact $\phi$ is the unique pseudo-Anosov flow 
transverse to $\FF$, perfect fits or not. 
More generally, we make the following conjecture.
\begin{conjecture}\label{rigidconj}
Let $\FF$ be a taut foliation of an atoroidal 3-manifold $M$ which is not $\R$-covered. If $\FF$ is almost transverse to a pseudo-Anosov flow $\phi$ with no perfect fits, then $\phi$ is the unique pseudo-Anosov flow almost transverse to $\phi$.
\end{conjecture}

We also propose the following general strategy toward \Cref{qclassify}. First, given a taut foliation $\FF$ of an irreducible, atoroidal 3-manifold $M$, let $\mathrm{PA}(\FF)$ be the set of all pseudo-Anosov flows that are almost transverse to $\FF$, up to orbit equivalence. Each such flow $
\phi$ gives rise to a faithful $\pi_1$-action on the boundary of its orbit space $\orb_\phi$, as shown by Fenley. Also, let $\mathrm{UC}(\FF)$ be the set of $\pi_1(M)$ circle actions arising from universal circles of $\FF$, considered up to semiconjugacy.

\begin{conjecture}\label{circleconj}
  Let $\FF$ be a taut foliation of an atoroidal 3-manifold $M$ which is not $\R$-covered.
The assignment $\varphi \mapsto (\pi_1(M)\curvearrowright\partial \orb_\phi)$ from \Cref{maintheorem} determines a bijection
\[
\mathrm{PA}(\FF) \to \mathrm{UC}(\FF).
\]
\end{conjecture}

We remark that since $\FF$ is not $\R$-covered, any almost transverse flow $\varphi$ is not $\R$-covered (\cite[Theorem C]{fenley2005regulating}), and so the action $\pi_1(M) \curvearrowright \partial \orb$ is minimal in the dynamical sense (i.e. every orbit is dense) by \Cref{prop:minimal}. Moreover, by a result of Ghys, the notions of conjugacy and semiconjugacy agree for dynamically minimal circle actions (see e.g. \cite[Proposition 4.8]{bucher_semiconj}).

\section{Foliations, flows, and circles}

\subsection{Taut foliations and universal circles}\label{sec:tautfols}
Let $\FF$ be a taut foliation of a closed, oriented, irreducible 3-manifold $M$. 
By this we mean $\FF$ is a cooriented foliation of $M$ by surfaces, such that there exists a closed curve positively transverse to $\FF$ and intersecting every leaf. This simple condition on $\FF$ has some nice consequences. By Novikov's theorem \cite{Nov65}, any closed curve transverse to $\FF$ is homotopically nontrivial. Further, each leaf of $\FF$ is $\pi_1$-injective so $\FF$ lifts to a foliation $\wt \FF$ of $\wt M$ by planes.

Preimages of leaves of $\FF$ under the universal covering map foliate the universal cover $\wt M$ of $M$; we denote this foliation of $\wt M$ by $\wt \FF$. If we collapse the leaves of $\wt \FF$, we obtain the \emph{leaf space} of $\FF$, denoted $\LL$. This is a (typically non-Hausdorff) 1-manifold.

If we additionally assume that $M$ is \emph{atoroidal}, meaning that $\pi_1(M)$ contains no $\Z\oplus \Z$ subgroup, then 
a powerful theorem of Candel concerning Riemann surface laminations implies that there is a Riemannian metric $g$ on $M$ which pulls back to a hyperbolic metric on each leaf of $\FF$ \cite{Candel_uniformization}. 
We call such a metric $g$ a \emph{leafwise hyperbolic metric}. 
This endows each leaf of $\LL$ with a circle at infinity. In fact, these circles at infinity are well-defined independently of the choice of $g$: for any other metric $g'$, the ratio $g'/g$ is uniformly bounded by compactness of $M$ and so the induced metrics on any leaf of $\wt \FF$ are uniformly quasi-isometric.

The union of these ideal circles is the total space $E_\infty$ of a circle bundle over $\LL$ which we will topologize and describe further in \Cref{sec:continuity}.

\subsection{Pseudo-Anosov flows and taut foliations}\label{sec:pAflows}

A \emph{pseudo-Anosov flow} $\phi$ on our closed oriented 3-manifold $M$ is a flow that is Anosov on the complement of finitely many closed \emph{singular orbits}, each of which has a neighborhood modeled on a branched cover of a closed orbit of an Anosov flow. In particular there are two singular foliations $W^s$ and $W^u$ called the \emph{stable} and \emph{unstable} foliations, respectively. Orbits in the same stable leaf converge in forward time, and orbits in the same unstable leaf converge in backward time. For a great discussion of the precise definition of pseudo-Anosov flow, see \cite{AgolTsang}.

\begin{figure}
\centering
\includegraphics[height=2in]{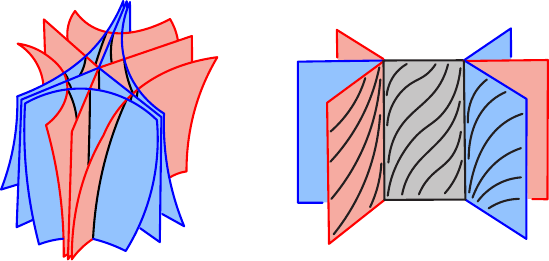}
\caption{Left: a 3-pronged singular orbit of a pseudo-Anosov flow (the flow is upward, the stable foliation is red, and the unstable is blue). Right: a dynamic blowup of a 3-pronged singular orbit. If more prongs are present, then more complex blowups are possible.}
\label{fig:blowup}
\end{figure}

It is convenient for us to expand the class of pseudo-Anosov flows by considering \emph{almost pseudo-Anosov flows}. An almost pseudo-Anosov flow is the result of replacing finitely many singular orbits with flow-invariant annulus complexes in a specific way, an operation called ``dynamic blowup." See \Cref{fig:blowup}. Each annulus in the blowup complex is called a \emph{blowup annulus}.
If $\phi'$ is an almost pseudo-Anosov flow obtained from a pseudo-Anosov flow $\phi$ in this way, we say $\phi'$ is a \emph{dynamic blowup of $\phi$}. Finally, we say a codimension one foliation $\FF$ of $M$ is \emph{almost transverse} to $\phi$ if it is transverse to some dynamic blowup of $\phi$. For more details see \cite{Mos96} or our treatment in \cite{LMTstrongtst}.

Mosher proved that an (almost) pseudo-Anosov flow on an atoroidal 3-manifold is always transitive, meaning it has a dense orbit \cite[Prop 2.7]{Mos92} (see also \cite[Corollary 1.6]{barthelme2024non}). As such, if $\FF$ is a foliation of $M$ almost transverse to a pseudo-Anosov flow, then it is taut, because we can find a closed transversal to all leaves by closing up a long flow segment.

\subsection{The flow space and its boundary}\label{sec:flowspace}

Let $\phi$ be an almost pseudo-Anosov flow on $M$.
Let $\wt \phi$ be the lift of $\phi$ to the universal cover $\wt M$ of $M$.
The stable and unstable foliations $W^s$ and $W^u$ lift to singular foliations $\wt W^s$ and $\wt W^u$ in $\wt M$. A \emph{half leaf} of a leaf $L$ of $W^{s/u}$ containing the lift of a periodic orbit is the closure of a component of $L-(\text{all lifts of periodic orbits})$.

Let $\orb$ be the quotient of $\wt M$ by the flowlines of $\wt \phi$, which we call the \emph{flow space} of $\phi$. Fenley-Mosher proved that $\orb$ is a plane \cite[Proposition 4.2]{fenley2001quasigeodesic}. The 2-dimensional foliations $\wt W^s$ and $\wt W^u$ project to 1-dimensional singular foliations of $\orb$ that we denote $\orb^s$ and $\orb^u$, respectively. We call these the stable and unstable foliations of $\orb$. The half leaves of this foliation are the projections of the half leaves of $\wt W^{s/u}$.
The action of $\pi_1(M)$ by deck transformations descends to an action of $\pi_1(M)$ on $\orb$ by orientation-preserving homeomorphisms that preserves the stable and unstable foliations.

Now let $\wt \FF$ be the lift of $\wt \FF$ to $\wt M$, and let $\lambda$ be a leaf of $\wt\FF$. The intersection of $\lambda$ with leaves of $\wt W^s$ and $\wt W^u$ define two singular foliations of $\lambda$ that we denote by $\lambda^s$ and $\lambda^u$.
A \emph{foliation ray} of $\orb^{s/u}$ or $\lambda^{s/u}$ is a proper embedding of $[0,\infty)$ into a leaf of the corresponding singular foliation.
A \emph{perfect fit rectangle} in $\orb$ is the image of a proper embedding of $([0,1]\times[0,1])-(1,1)$ into $\orb$, such that vertical and horizontal lines are mapped into leaves of the stable and unstable foliations. We say the foliation rays corresponding to $\{1\}\times [0,1)$ and $[0,1)\times\{1\}$ \emph{make a perfect fit}, and more generally we say that two foliation rays make a perfect fit if they have subrays lying in the boundary of a perfect fit rectangle as above. See \Cref{fig:perfectfits}.
Intuitively these rays meet ``at infinity," and one makes this precise by defining an ideal boundary for $\orb$.

\begin{figure}
\begin{center}
\includegraphics[height=1.25in]{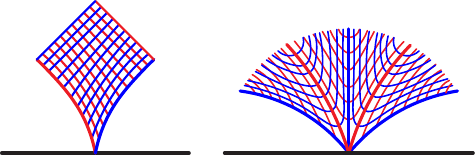}
\caption{In both pictures, the horizontal line denotes $\del\orb$ and $\orb$ lies above it. Left: two foliation rays making a perfect fit, and an associated perfect fit rectangle. Right: any two foliation rays in $\del\orb$ with the same endpoint in $\del\orb$ are part of a finite sequence of foliation rays ending at that point, each making a perfect fit with the next.}
\label{fig:perfectfits}
\end{center}
\end{figure}

In \cite{fenley2012ideal}, Fenley shows that the flow space has a natural compactification
by a circle $\del\orb$ to a closed disk $\ol \orb$. Moreover, the action of $\pi_1(M)$ on
$\orb$ extends to one on $\ol\orb$, which in turn restricts to a faithful action 
\[
\rho_\phi\colon \pi_1(M)\to \Homeo^+(\del\orb)
\]
In $\ol \orb$, every foliation ray limits on a unique point, and it is a property of the
construction that the endpoints of foliation rays are dense in $\del\orb$. Moreover, two foliation rays in
$\orb$ have the same endpoint $p\in \del\orb$ if and only if they are part of a finite sequence of rays ending at $p$, each making a perfect fit with the next (\cite[Lemma 3.20]{fenley2012ideal}), see \Cref{fig:perfectfits}.

A fundamental fact about the boundary $\del\orb$ is that if two leaves of $\orb^s$ and/or $\orb^u$ share an ideal point in $\del \orb$, then they are disjoint in $\orb$. An abstract characterization of $\partial \orb$ that includes all of the required properties can also be found in \cite[Theorem 3]{bonatti2301action}.

\subsection{Shadows of leaves}

We now mention some background on almost pseudo-Anosov flows transverse to foliations, the reference for which is \cite[\S\S 3-4]{Fen09}.
We will work in the setting where $\phi$ is an almost pseudo-Anosov flow transverse to a foliation $\FF$ of $M$.

Projection to the flow space $\orb$ carries a leaf $\lambda$ of $\wt\FF$ homeomorphically onto its image, which we denote by $\Omega_\lambda$. We sometimes call this open set the \emph{shadow} of $\lambda$. The topological frontier of $\Omega_{\lambda}$ in $\orb$ is denoted $\fr \Omega_{\lambda}$. The closure of $\Omega_\lambda$ in $\ol \orb$ is obtained by taking the union of $\Omega_\lambda$ with $\fr \Omega_{\lambda}$ and its limit set $\partial_\infty \Omega_{\lambda} \subset \partial \orb$, and we denote this closure by $\ol  \Omega_{\lambda}$.

A leaf $L$ of $\orb^s$ or $\orb^u$ is either homeomorphic to $\R$ or homeomorphic to compact simplicial tree with at least one vertex of degree $\ge3$, without its degree 1 vertices.
In the first case we say that $L$ is \emph{regular}, and in the second we say that $L$ is \emph{singular}. A point $p\in\orb$ is a \emph{singular point} if all leaves of $\orb^s$ and $\orb^u$ through $p$ are singular and is otherwise a \emph{regular point}.

Given a leaf $L$ as above, a \emph{leaf face} $\ell$ in $L$ is a properly embedded copy of $\R$ in $L$ with the property that at least one of the two components of $\orb -\ell$ contains no other points of $L$. Be advised that Fenley calls a leaf face a ``line leaf."

A \emph{leaf slice} of $L$ is a properly embedded copy of $\R$ in $L$. 
Let $\ell$ be a leaf slice of a singular stable leaf $L$, and let $L'$ be the corresponding leaf of the unstable foliation. If $\ell$ has a side such that a single unstable half leaf (this could be a \emph{blowup segment}, i.e. the projection of the lift of a blowup annulus) emanates from $\ell$ into that side, we say $\ell$ is \emph{regular} to that side. 
The symmetric definition holds for stable leaf slices, and we also declare that regular leaves are regular to both of their sides. A general leaf slice may be regular to zero, one, or two of its sides. 
Some of these definitions are illustrated in \Cref{fig:leaves}.

\begin{figure}
\begin{center}
\includegraphics[height=1.35in]{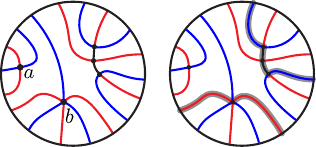}
\caption{Two pictures of $\orb$ with some example leaves of $\orb^{s/u}$. The circle is $\del\orb$. Left: $a$ is a regular point, and $b$ is a singular point. The black segments are blowup segments. Right: two leaf slices. The lower one is regular only to its top side. The upper one is regular to neither of its sides.
}
\label{fig:leaves}
\end{center}
\end{figure}

Returning to our leaf $\lambda$ of $\wt \FF$ and its shadow $\Omega_\lambda\subset\orb$, Fenley shows that $\fr\Omega_\lambda$ is a disjoint union of leaf slices, each of which is regular to the side containing $\Omega_\lambda$. See \Cref{fig:shadow}.

\begin{lemma}\label{itsadisk}
The closure $\ol  \Omega_{\lambda}$ is a disk in $\ol \orb$ that meets $\partial \orb$ in $\partial_\infty \Omega_{\lambda}$ and $\orb$ in $\Omega_{\lambda} \cup \fr \Omega_{\lambda}$. 
\end{lemma}

\begin{proof}
The only thing to show is that $\ol \Omega_\lambda$ is a disk. We first claim that $\ol  \Omega_{\lambda} - \Omega_\lambda$ is a circle. For every component of $\fr{\Omega_\lambda}$, define a homeomorphism between the component and the corresponding segment in $\del\orb$. Then glue all these maps together to get a map $\ol  \Omega_{\lambda}-\Omega_\lambda\to \del \orb$. The result is continuous by the pasting lemma and the inverses for the restrictions to the frontier components glue to give a continuous inverse. Now $\ol \Omega_\lambda$ is a disk by the Schoenflies theorem. 
\end{proof}

Moving forward, we use $\partial \ol \Omega_\lambda$ to denote the boundary of $\ol \Omega_\lambda$, which is of course a circle.  As in \Cref{itsadisk}, it is disjoint union of the limit set $\partial_\infty \Omega_\lambda\subset \del\orb$ and the frontier $\fr \Omega_\lambda\subset \orb$.

\subsection{Circular orders}
\label{sec:order}
Let $X$ be a set equipped with a map
\[
\langle \cdot,\cdot, \cdot\rangle\colon X^3\to \{+1, 0, -1\}. 
\]
One should think of the values in the range as specifying clockwise, neutral, and anticlockwise. Suppose that 
\begin{enumerate}[label=(\alph*)]
\item$\langle a,b,c\rangle\ne 0$ precisely when $a,b,c$ are all distinct, and 
\item for any $a,b,c,d\in X$, $\langle b,c,d\rangle-\langle a,c,d\rangle+\langle a,b,d\rangle- \langle a,b,c\rangle=0$ (the cocycle condition).
\end{enumerate}
Then we say that $\langle\cdot, \cdot, \cdot\rangle$ is a \emph{circular order} on $X$, or sometimes simply that $X$ is \emph{circularly ordered} if the circular order is implied.
Given that we wish to think of $X$ as similar to a circle, condition (b) is natural: it says in particular that for any nondegenerate 3-simplex $(a,b,c,d)$ in the abstract simplex spanned by $X$, two of its (oriented) faces are oriented clockwise and two are oriented anticlockwise; see \cite[Def. 2.1.2 and Fig. 1]{thurston_circles2}.   

The triple $(a,b,c)$ is called \emph{positive}, \emph{negative}, or \emph{degenerate} if $\langle a,b,c\rangle$ is $+1$, $-1$, or $0$ respectively.
The open interval $(a,b)$ is the set of points $x\in X$ such that $\langle a, x, b\rangle=+1$. Replacing a parenthesis by a bracket denotes the inclusion of the corresponding endpoint.

If $S^1_X$ is an oriented circle, then there is a canonical circular order defined on triples of distinct points $(a,b,c)$ by $\langle a,b,c\rangle=+1$ if the oriented arc from $a$ to $c$ contains $b$, or else  $\langle a,b,c\rangle=-1$. We will often use this order implicitly in what follows.

Finally, suppose $Y$ is another circularly ordered set. A map $f\colon X\to Y$ is \emph{monotone} if it sends no nondegenerate triple to one of the opposite sign. When $X$ and $Y$ are oriented circles, a continuous map between them is monotone if and only if it has degree 1 and its point preimages are contractible.
Given a monotone map of circles $m\colon S^1_X\to S^1_Y$, a \emph{gap} of $m$ is a maximal open connected interval in $S^1_X$ sent by $m$ to a single point. The \emph{core} of $m$ is the complement of its gaps.

\section{Monotone maps from bifoliations}\label{sec:monotonemaps}
We continue to work in the following setting: $M$ is an atoroidal 3-manifold, $\phi$ is an almost pseudo-Anosov flow on $M$ transverse to a taut foliation $\FF$, and we have fixed a leafwise hyperbolic metric on $M$ that restricts to a hyperbolic metric on each leaf of $\FF$.
In this subsection we fix a leaf $\lambda$ of $\wt \FF$, and denote its hyperbolic boundary by $\del \lambda$.

The two planes $\lambda$ and $\Omega_{\lambda}$ are equipped with pairs of singular foliations $\lambda^{s/u}$ and $\Omega_\lambda^{s/u}$, respectively. Moreover, the projection from $\wt M$ to $\orb$ restricts to a homeomorphism of these planes carrying $\lambda^{s/u}$ to $\Omega_\lambda^{s/u}$. 

Studying the ideal points of the leaves of these foliations will lead naturally to the existence of a monotone map $\del\ol \Omega_\lambda\to \del \lambda$ between the planes' boundaries. 
We begin with an observation about stable and unstable foliation rays in $\Omega_\lambda$.

\begin{observation}[Rays in a shadow]\label{obs:shadow}
Foliation rays in $\Omega_\lambda$ have well-defined endpoints in $\partial \ol \Omega_\lambda$ and the endpoints of all foliation rays (stable and unstable together) are dense, except along blow up segments in the frontier. 
\end{observation}

Next, we record some fundamental properties of foliation rays in $\lambda$, due to Fenley.

\begin{lemma}[Rays in a leaf, Fenley]\label{lem:omniFenley}
Foliation rays in $\lambda$ exhibit the following properties:
\begin{enumerate}[label=(\roman*)]
\item \cite[Corollary 5.5]{Fen09}: Foliation rays in $\lambda$ have well-defined endpoints in $\del\lambda$. 
\item \cite[Prop 6.3]{Fen09}: The endpoints of stable foliation rays are dense in $\del\lambda$. The same is true for unstable foliation rays.
\end{enumerate}
\end{lemma}

The sets of foliation rays of $\lambda$ and $\Omega_\lambda$ are canonically identified, and we denote the set of ends of these rays by $\mc E$. 
There is a natural circular order on $\mc E$: the ends of triples of distinct rays are assigned $\pm 1$ by considering the rays' intersections with the boundaries of large disks. The key point is that any two leaves of the bifoliation have compact intersection. For details see \cite[Chapter 4]{Frankel_thesis} or \cite[Section 2]{bonatti2301action}.

Now let $\mc E_\lambda$ and $\mc E_\Omega$ be the subsets of $\partial \lambda$ and $\partial \ol \Omega_\lambda$ consisting of ideal points of stable and unstable foliation rays. Each of these sets is embedded in an oriented circle and hence is circularly ordered.
Since $\partial \lambda$ and $\partial \ol \Omega_\lambda$ are compactifications of isomorphic bifoliated planes, and stable/unstable foliation rays have well-defined endpoints in both, we have maps
\begin{center}
\begin{tikzcd}
&\mc E\arrow{ld} \arrow{rd}&\\
\mc E_\Omega&&\mc E_\lambda
\end{tikzcd}
\end{center}
which are evidently monotone. The next lemma says that the above diagram has a natural completion to a commutative triangle of monotone maps.

\begin{lemma} \label{lem:single_valued}
If two foliation rays $r_1,r_2$ have the same endpoints in $\partial \ol  \Omega_{\lambda}$, then they have the same endpoints in $\partial\lambda$. 
Hence, the maps above induce a monotone map $\mc E_\Omega \to \mc E_\lambda$.
\end{lemma}

\begin{proof}
If the rays are distinct, they must end at a point in $\partial_\infty
\Omega_\lambda\subset\partial\orb$. Hence they are joined by a sequence of perfect fits (see \Cref{sec:flowspace}). We now observe that two rays making a perfect fit must have the same endpoints in $\partial_\infty \lambda$: there are no leaves of either foliation in between them, so the endpoints must be the same by density of endpoints (\Cref{lem:omniFenley}(ii)). This induces a map $\mc E_\Omega \to \mc E_\lambda$ completing the above diagram to a commutative triangle. This map is monotone because the other two maps are.
\end{proof}

Let $S^1_X, S^1_Y$ be oriented circles, let $ A\subset S^1_X$, and let $ B\subset S^1_Y$ be dense.
It is a well known exercise that a monotone surjection $f\colon A\to  B$ induces a unique continuous monotone map $F\colon S^1_X\to S^1_Y$ (necessarily defined at limit points by $F(\lim(x_n))=\lim(f(x_n))$ where $(x_n)$ is any Cauchy sequence of points in $A$).

\medskip

Hence we obtain a unique continuous monotone map $\partial \ol \Omega_\lambda \to \partial \lambda$ that extends the one in \Cref{lem:single_valued}. 
Fenley's \Cref{lem:omniFenley} gives a quick proof that this map collapses each leaf in the frontier of $\Omega_\lambda$:

\begin{lemma}\label{frontiercollapse}
If two foliation rays in $\Omega_\lambda$ have endpoints in the same component of $\fr \Omega_\lambda$, then they have the same endpoint on $\partial \lambda$.
\end{lemma}

\begin{proof}
The endpoints of leaves terminating at a component of $\fr \Omega_{\lambda}$ give an arc in $\partial \lambda$ containing no endpoints of the other foliation. By density of endpoints (\Cref{lem:omniFenley}(ii)) it must be a trivial arc, i.e. a point. 
\end{proof}

\begin{construction}[Monotone maps $\pi_\lambda \colon \partial \orb \to \partial \lambda$]\label{monotonemaps}
Since the monotone map $\del\ol\Omega_\lambda\to \del \lambda$ collapses the closure of each frontier leaf of $\Omega_\lambda$ to a point by \Cref{frontiercollapse}, it induces a unique monotone map on $\partial \orb$ that we denote by 
\[
\pi_\lambda\colon \del\orb \to \partial \lambda.
\]
In more detail,
choose a homeomorphism $h\colon\del\ol\Omega_\lambda\to \orb$ fixing $\del_\infty \ol\Omega_\lambda$, as in the proof of \Cref{itsadisk}. 
Precomposing $\del\ol\Omega_\lambda\to \del \lambda$ with $h^{-1}$ gives the 
required monotone $\pi_\lambda\colon \del\orb \to \partial \lambda$
since it is independent of the choice of $h$ by \Cref{frontiercollapse}. 

If we replace $\phi$ by its associated pseudo-Anosov flow, $\del\orb$ is unchanged up to a canonical equivariant identification. Hence we can view the maps $\pi_\lambda$ as having domain equal to the boundary of the orbit space of this pseudo-Anosov flow.
\end{construction}

We will see in \Cref{sec:fibers} that $\pi_\lambda$ collapses only finite chains of frontier components of $\Omega_\lambda$ that meet at ideal points in $\del\orb$.

\section{Continuity and the circle bundle at infinity}\label{sec:continuity}

Recall from \Cref{sec:tautfols} that $\LL$ denotes the \emph{leaf space} of $\wt\FF$, i.e. the quotient of $\wt M$ by the
leaves of $\FF$. 
The \emph{circle bundle at infinity} associated to $\FF$ is a circle bundle 
\[
E_\infty\to \LL,
\]
the fiber over $\lambda\in \LL$ being the hyperbolic boundary $\del \lambda$ of $\lambda$ coming from a leafwise hyperbolic metric. 
Given a leaf $\lambda$ of $\wt \FF$, and $p\in \lambda$, the endpoint map $e_p\colon UT_p(\lambda)\to \del\lambda$ maps a unit tangent vector at $p$ to the ideal endpoint of the geodesic $\gamma$ in $\lambda$ with $\gamma(0)=p$ and $\gamma'(0)=v$.
Letting $\lambda$ and $p$ vary, we obtain a map
\[
e\colon UT(\wt \FF)\to E_\infty.
\]
We endow $E_\infty$ with the finest topology such that $e$ is continuous.
As Calegari points out in \cite[\S 7.2]{2007CalegariBook}, for any interval $I$ transverse to $\wt \FF$, the map $e_\infty$ restricts to a local homeomorphism $UT(\wt \FF)|_I\to E_\infty$. This gives $E_\infty$ an atlas of cylindrical charts $I\times S^1$. 

\medskip

In 
\Cref{monotonemaps} 
we built, for any $\lambda\in \LL$, a surjection $\pi_\lambda\colon\del \orb \to \del \lambda$. Letting $\lambda$ vary, we obtain a surjection
\[
\pi\colon \LL\times \del\orb\to E_\infty
\]
defined by $\pi(\lambda,p) = \pi_\lambda(p)$.

We will prove that $\pi$ is continuous. The key ingredient to this is a lemma of Fenley saying that endpoints of rays in $E_\infty$ vary continuously with $\lambda\in\LL$ \cite[Lemma 6.2]{Fen09}:
\begin{lemma}[Fenley]\label{endpointsvarycontinuously}
Let $W$ be a leaf of $\wt W^s$ (or $\wt W^u$), and let $r$ be a foliation ray in $W\cap \lambda$ containing no singular points. 
There is an open neighborhood $U$ of $\lambda$ in $\mc L$ so that for each $\mu \in U$, $W\cap \mu$ contains a foliation ray $r_\mu$ near $r$ with the property that the map from $U$ to $E_\infty$ sending $\mu$ to the ideal endpoint $r_\mu(\infty) \in \del \mu$ of $r_\mu$ is continuous.
\end{lemma}

Using the definition of $\pi_\lambda$, the following is immediate: if $x \in \partial \orb$ is the ideal endpoint of a stable/unstable ray $r$ in $\orb$ that meets $\Omega_\lambda$ for some $\lambda \in \mc L$, then $\mu \mapsto \pi_\mu(x)$ defines a continuous map from some $\LL$-neighborhood of $\lambda$ into $E_\infty$. 

This observation allows us to promote the continuity of leafwise maps to continuity of $\pi$:

\begin{proposition}\label{continuity}
The map $\pi\colon \LL\times \del \orb\to E_\infty$ is continuous.
\end{proposition}

\begin{proof}
Fix any $(\lambda,p)\in \mc L\times \del \orb$ and a sequence $(\lambda_n, p_n)$ converging to $(\lambda, p)$. We will show $\pi_{\lambda_n}(p_n)\to \pi_\lambda (p)$.

Let $\mc E_\orb\subset \del \orb$ be the dense set of endpoints of rays in $\del \orb$. Choose $x,y\in\mc E_\lambda$ so that $\pi_\lambda(p)\in (x,y)$. Let $a$ and $b$ be points in $\mc E_\orb$ that are mapped by $\pi_\lambda$ to $x$ and $y$, respectively. By \Cref{endpointsvarycontinuously}, we have $\pi_{\lambda_i}(a)\to x$ and $\pi_{\lambda_i}(b)\to y$. 

 Note that $p\in(a,b)$ since $\pi_\lambda$ is monotone, so the $p_i$
 eventually lie in $(a,b)$. Since each $\pi_{\lambda_i}$ is monotone, we have $\pi_{\lambda_i}(p_i)\in [\pi_{\lambda_i}(a), \pi_{\lambda_i}(b)]$ for all $i$. This forces any accumulation point of $(\pi_{\lambda_i}(p_i))_{i\ge0}$ to lie in $[x,y]$. Since we are free to choose $x$ and $y$ as close as we like to $\pi_{\lambda}(p)$, this proves $\pi_{\lambda_n}(p_n)\to \pi_\lambda(p)$. Hence $\pi$ is continuous at $(p,\lambda)$.
\end{proof}

We remark that the powerful \emph{Leaf Pocket Theorem} of Thurston and Calegari--Dunfield (\cite[Theorem 5.2]{CalDun_UC} plays an important role in the background of \Cref{continuity},  because Fenley's proof of \Cref{endpointsvarycontinuously} uses the result in an essential way. The Leaf Pocket Theorem says that a leaf of $\wt \FF$ stays close to other leaves of $\wt \FF$ along a dense set of directions, and is the tool that allows Fenley to transport rays to nearby leaves in a controlled fashion.

\section{Simultaneous universal circles} \label{sec:simuc}
We continue to work with a taut foliation $\FF$ of an atoroidal manifold $M$ equipped with a leafwise hyperbolic metric, and an almost pseudo-Anosov flow $\phi$ that is transverse to $\FF$.

Before defining universal circles, we remind the reader that, given a monotone map of circles $m\colon S^1_X\to S^1_Y$, a \emph{gap} of $m$ is a maximal open connected interval in $S^1_X$ sent by $m$ to a single point. The \emph{core} of $m$ is the complement of its gaps.

\begin{definition}\label{UCdef}
A \emph{universal circle} for $\FF$ is a triple $(\univ, \{m_\lambda\},\rho)$ consisting of a circle $\univ$  and a family of monotone maps $\{m_\lambda\colon  \univ\to \del \lambda \mid \lambda\in \LL\}$, together with a faithful representation
\[
\rho\colon \pi_1(M)\to \Homeo_+(\univ)
\]
such that the following hold.
\begin{enumerate}
\item The map $m\colon \LL\times \univ \to E_\infty$ defined by $m(\lambda, p)=m_\lambda(p)$ is continuous.
\item For each $\lambda\in \LL$ and $g\in \pi_1(M)$, the following diagram commutes:
\begin{center}
\begin{tikzcd}
S^1_{\text{univ}} \arrow{r}{\rho(g)}\arrow{d}{m_\lambda}&S^1_{\text{univ}}\arrow{d}{m_{g\lambda}}\\
\del\lambda\arrow{r}{g}&\del(g \lambda)
\end{tikzcd}
\end{center}
\item If $\lambda, \lambda'$ are leaves of $\wt\FF$ not connected by a transversal, then the core of $\pi_\lambda$ is contained in a single gap of $\pi_{\lambda'}$.\qedhere
\end{enumerate}
\end{definition}

Thurston and Calegari--Dunfield proved that every taut foliation of a closed, oriented, atoroidal 3-manifold has a universal circle \cite[Theorem 6.2]{CalDun_UC}. 

The original definition of a universal circle was given by Thurston in a series of lectures at MSRI. According to Calegari, that definition was less axiomatic and did not feature condition (3) \cite[p. 263]{2007CalegariBook}. The next definition was given by Calegari--Dunfield in \cite{CalDun_UC} and agrees with \Cref{UCdef} apart from not requiring condition (1). Despite the slight differences, the Thurston and Calegari--Dunfield construction satisfies \Cref{UCdef}, which is equivalent to the definition in \cite{2007CalegariBook}.

\begin{theorem}\label{simUC}
Let $M$ be a closed oriented atoroidal manifold and $\phi$ a pseudo-Anosov flow almost transverse to a foliation $\FF$ of $M$.
The circle $\del \orb$, together with the faithful representation $\rho_\phi\colon\pi_1(M)\to \Homeo^+(\del\orb)$ and the monotone maps $\pi_\lambda\colon \del\orb\to \del \lambda$ from \Cref{monotonemaps}, comprise a universal circle for $\FF$.
\end{theorem}

\begin{proof}
In this setting condition (1) of \Cref{UCdef} is a restatement of \Cref{continuity}.
For condition (2), given $g\in \pi_1(M)$ and $\lambda\in \LL$, the diagram commutes for points in $\mc E_\orb$ and continuity does the rest.
Finally, for (3) note that if there is no transversal from $\lambda$ to $\lambda'$, then the shadows $\Omega_\lambda$ and $\Omega_{\lambda'}$ are disjoint, so the core of each is contained in a gap of the other.
\end{proof}

\section{Fibers of the monotone maps}\label{sec:fibers}
We are still fixing a closed oriented atoroidal 3-manifold $M$, a pseudo-Anosov flow $\phi$, and a foliation $\FF$ almost transverse to $\phi$.

Our goal in this section is to characterize the gaps of the monotone maps $\pi_\lambda$ from \Cref{simUC} as follows: each gap is ``spanned by a finite chain of frontier components of $\Omega_\lambda$." We will make this precise.

\subsection{The flow space trichotomy}

A pseudo-Anosov flow is called \emph{$\R$-covered} if the leaf space of $\orb^s$ (or $\orb^u$) is homeomorphic to $\R$. In this case the flow is actually \emph{Anosov} and not just pseudo-Anosov. By results of Fenley \cite{fenley1994anosov} and Barbot \cite{barbot1995caracterisation} the structure of $(\orb, \orb^s, \orb^u)$ is either:
\begin{enumerate}
\item \emph{trivial}: there is a homeomorphism $\orb\to \R^2$ carrying $\orb^s$ and $\orb^u$ to the foliations of $\R^2$ by vertical and horizontal lines, or
\item \emph{skew}: there is a homeomorphism of $\orb$ with the diagonal strip $\{(x,y)\mid |x-y|< 1\}$ carrying $\orb^s$ and $\orb^u$ to the foliations by vertical and horizontal lines.
\end{enumerate}
In the trivial case, the flow is orbit-equivalent to the suspension of an Anosov diffeomorphism of the torus \cite{barbot1995caracterisation}. In the skew case we say the flow is \emph{skew Anosov}. Otherwise, the flow is 
\begin{enumerate}[resume]
\item \emph{non $\R$-covered}: one of $\orb^{s/u}$ has non-Hausdorff leaf space, or the flow has at least one singular orbit (perhaps both). 
\end{enumerate}
See \cite[Theorem 2.16]{BFM} for a proof of this trichotomy working in the more general setting of bifoliated planes satisfying certain axioms.

In the setting of this article, with our fixed objects $M,\FF,\phi$, the fact that $M$ is atoroidal implies that $\phi$ is either skew Anosov or non $\R$-covered. To this point it has not been necessary to distinguish between these two cases, but it will be so in this section.

\subsection{Lozenges and periodic points}\label{sec:lozenges}

We say $p\in \orb$ is \emph{periodic} if there is an element $g$ that fixes $p
\in\orb$. In this case, $p$ corresponds to a periodic orbit $\gamma$ of $\phi$. 

A \emph{lozenge} in $\orb$ is the image of a proper embedding of $(I\times I-\{(0,0), (1,1)\})\to \orb$ such that vertical and horizontal lines in $I\times I$ are mapped into stable and unstable leaves.
A point in a lozenge corresponding to $(1,0)$ or $(0,1)$ is called a \emph{corner point} of that lozenge. If one corner point of a lozenge is periodic, then so is the other, and they are fixed by a common group element. 
Moreover, the two corresponding orbits in $M$ are homotopic to each other's inverses (up to positive powers). 
If a periodic point is not the corner of any lozenge, we call it a \emph{noncorner periodic point}.

More generally, we will use the term lozenge to refer to any region of $\orb$ that blows down to a lozenge as above by collapsing blowup segments, and refer to the corners of a lozenge as the preimages of the corner points under the collapsing. In general, the corner of a lozenge will consist of either a single point or a union of blowup segments.

If $g$ fixes two points $p$ and $p'$ of $\orb$, then Fenley showed that $p$ and $p'$ are connected by a ``chain of lozenges" \cite[Theorem 3.3]{fenley1995quasigeodesic} (see also \cite[Proposition 2.24]{BFM}). That is, there is a finite collection of lozenges $L_1,\dots, L_n$ such that $p$ is a corner point of $L_1$, $p'$ is a corner point of $L_n$, and $L_i$ shares a corner point with $L_{i+1}$ for $1\le i\le n-1$.  In particular, if there are no blowup segments in $\orb$ then noncorner periodic points are exactly the periodic points that are their stabilizers' only fixed points.

If $g$ fixes the periodic point $p$, then up to possibly replacing $g$ by a higher power, we can assume that the action of $g$ fixes the stable and
unstable leaves through $p$ as well as their endpoints in $\del \orb$. There may be
other points of $\orb$ fixed by $g$, which by above are all corner points of a family of lozenges whose union is connected. 
The action of $g$ fixes all the lozenges and hence all the endpoints of leaves through the corner points.
If $\phi$ is non $\R$-covered, then the fixed point set of the action of $g$ on $\del \orb$ is exactly the closure of these endpoints (see e.g. \cite[Proposition 3.7]{BFM}). If $g$ fixes only finitely many points in $\orb$, then the fixed points of its action on $\del \orb$ are exactly the ideal endpoints of the stable and unstable leaves through the fixed points in $\orb$. Thus there are $2n\ge4$ fixed points, and one sees that they alternate between attracting and repelling. In this situation we say that $g$ acts with \emph{multi sink-source dynamics} on $\del \orb$.

Finally, if $g\in\pi_1(M)$ fixes a leaf $\ell$ and an endpoint $p$ of $\ell$, then 
every leaf that makes a perfect fit with $\ell$ at $p$ is fixed by a power of $g$ (\cite[Observation 2.10]{BFM}).
Hence, if two foliation rays $r$ and $r'$ in periodic leaves make a perfect fit, then the leaves containing $r$ and $r'$ are fixed by a common group element and hence connected by a chain of lozenges. Hence $r$ lies in a leaf containing a corner point.

\subsection{Branching leaves}\label{sec:branching}

We say that a leaf of $\orb^s$ or $\orb^u$ is a \emph{branching leaf} if it is nonseparated from another leaf in the leaf space of $\orb^s$ or $\orb^u$, respectively, and otherwise we say the leaf is \emph{nonbranching}. A fundamental result of Fenley says that branching leaves are periodic, and that there are finitely many branching leaves up to the action of $\pi_1(M)$ (\cite[Theorem 4.9]{fenley1999foliations}). 
Combined with the fact that $M$ is atoroidal, this will imply that any $g\in \pi_1(M)$ can preserve at most finitely many branching leaves. We will use this observation of Fenley repeatedly, so we record it as a lemma along with another useful fact whose proof is the same.

\begin{lemma}[Fenley]\label{zplusz}
Let $M$ be atoroidal. The action of $g\in \pi_1(M)$ on $\orb$ fixes at most finitely many branching leaves, and finitely many leaves containing blowup segments. Moreover, the number of such leaves fixed by $g$ is uniformly bounded, independently of $g$. 
\end{lemma}

\begin{proof}
Suppose that $g\in \pi_1(M)$ fixes infinitely many branching leaves. By Fenley's finiteness result above, there are two of them, $\ell_1$ and $\ell_2$, and some $h\in \pi_1(M)$ such that $\ell_1=h\cdot\ell_2$. Hence $hgh^{-1}\cdot\ell_1=\ell_1$, so $hgh^{-1}\in\Stab(\ell_1)$. Fixing a generator $\gamma$ for $\Stab(\ell_1)$, there exist integers $k$ and $\ell$ so that $g=\gamma^k$ and $hgh^{-1}=\gamma^\ell$, so 
\[
h\gamma^k h^{-1}=\gamma^\ell.
\] 
However, Theorem 1 of \cite{shalen2001three} implies that such a relation (with $\gamma$ infinite order) does not exist in an atoroidal $3$-manifold group. 
Uniform boundedness follows from the result of Fenley mentioned before the Lemma statement.

The proof for leaves containing blown segments is identical, since there are only finitely many leaves in $M$ containing blowup annuli. 
\end{proof}

In particular, the above implies that only finitely many rays may terminate at a given point in $\del\orb$.

Recall from \Cref{sec:lozenges} that if $g$ fixes a point in $\orb$ as well as the half leaves based at $p$, then every point in $\orb$ fixed by $g$ is a corner point of some lozenge, any two such points are connected by a chain of lozenges, the collection of all such lozenges is fixed by $g$, and the set of points in $\del\orb$ fixed by $g$ is exactly the closure of the set of endpoints of all half leaves through these corner points. \Cref{zplusz} allows us to prove that the limit points of this set are never endpoints of the half leaves:

\begin{lemma}\label{sink}
Let $M$ be atoroidal. If $g\in\pi_1(M)$ fixes a leaf $\ell$ and an endpoint $y$ of $\ell$, then $y\in \partial \orb$ is a sink or source for the action $\pi_1(M) \curvearrowright \partial \orb$.\end{lemma}

\begin{proof}
We already observed in \Cref{sec:lozenges} that isolated endpoints of leaves fixed by $g$ are sources or sinks, so it suffices to show that, in the case where $g$ fixes infinitely many lozenges, any accumulation point of endpoints of leaves fixed by $g$ is not the endpoint of such a leaf. 

Let $L_1,L_2,L_3,\dots$ be a chain of lozenges, so that adjacent terms intersect along a side or corner. The key observation is that there is some $n\ge0$ such that $L_i$ and $L_{i+1}$ intersect only in their corners for all $i\ge n$; otherwise, we would produce infinitely many branching leaves fixed by $g$, contradicting \Cref{zplusz}. See \Cref{fig:lozengechain}. 

\begin{figure}
\begin{center}
\includegraphics[]{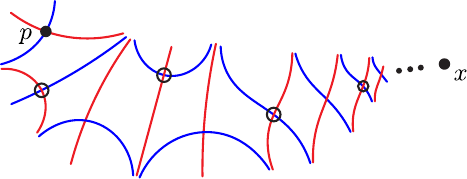}
\caption{From the proofs of \Cref{sink} and \Cref{prop:return}. For the latter: none of the circled points can lie in $\Omega_\lambda$.}
\label{fig:lozengechain}
\end{center}
\end{figure}

This implies that the set of all lozenges fixed by $g$ consists of finitely many lozenges that intersect along sides, and finitely many sequences of lozenges intersecting only along corners. it follows that the endpoints of leaves fixed by $g$ are isolated fixed points.
\end{proof}

If $r$ and $r'$ are both stable, or both unstable, rays that end at the same point in $\del\orb$ and are contained in leaves that are nonseparated from each other, 
 then there is a third ray $r''$ contained in a leaf of the other foliation that makes a perfect fit with both (\Cref{sec:flowspace}). Further, there are two lozenges containing $r, r', r''$ in their frontiers and meeting along the leaf containing $r''$. In particular, this implies that if $p\in \orb$ is a noncorner periodic point, then the stable and unstable leaves through $p$ are not branching.

If $(\ell_n)$ is a sequence of leaves of $\orb^s$ or $\orb^u$ limiting to a \emph{single} leaf slice $\ell$, then the endpoints of the $\ell_n$ limit to the endpoints of $\ell$ (see \cite[Proposition 3.1]{BFM}). This implies that if $\ell$ is a leaf face in a nonbranching leaf, then the leaves which lie sufficiently close to $\ell$ on the side to which it is regular have endpoints very close to the endpoints of $\ell$.

\subsection{Minimality of the boundary action}

The following proposition was proved by Fenley in the case of pseudo-Anosov flows with no perfect fits and by Bonatti for Anosov flows (\cite[Main Theorem]{fenley2012ideal}, \cite[Theorem 5]{bonatti2301action}). After circulating an earlier version of this article, we learned that it was also  proved recently by Barthelmé--Bonatti--Mann in an even more general form (\cite[Theorem 7.3]{barthelme2024non}).

\begin{proposition} [Barthelmé--Bonatti--Mann]\label{prop:minimal}
Let $\phi$ be a pseudo-Anosov flow on a 3-manifold $M$ which is not necessarily atoroidal. If $\varphi$ is not $\RR$--covered, then the action $\pi_1(M) \curvearrowright \partial \orb$ is minimal.
\end{proposition}

\begin{proof}
Fix $p\in \partial\orb$ and an open interval $I \subset \partial\orb$. The set of branching leaves is countable, and it follows that the endpoints of nonbranching leaves are dense in $\partial \orb$; see e.g. \cite[Theorem 3]{bonatti2301action}.
Pick a ray $r$ in a nonbranching leaf with $r(\infty) \in I$. Since $\orb$ is neither trivial
nor skew, the noncorner periodic points are dense in $\orb$ by \cite[Lemma
  2.30]{BFM}. 
Take a sequence of distinct noncorner periodic points $(x_i)$ converging to a point $q$ in the interior of $r$, and let $g_i\in \pi_1(M)$ be a generator of the stabilizer of $x_i$. Since the $x_i$ are noncorner points, $\langle g_i \rangle \neq \langle g_j \rangle$ for $i\ne j$ (see \Cref{sec:lozenges}).

By construction, initial segments of foliation rays $r_i$ based at the $x_i$ converge to the subray of $r$ based at $q$, and since $r$ is regular, we have that $r_i(\infty) \to r(\infty)$. 
%\ycom{is this part of our description of Fenley's boundary for $\orb$?} 
%\mcom{yes, see \Cref{sec:branching}}
Hence we can find sufficiently large $j>i$ so that the foliation rays through $x_i$ and $x_j$ end in $I$. We conclude that $g_i$ and $g_j$ act with multi sink-source dynamics, fixing only the endpoints of the foliation rays based at $x_i$ and $x_j$ (e.g. by \cite[Prop. 3.7 and Lem. 3.8]{BFM}), and that each of these elements has a fixed point in $I$. %From this, one sees that the point $p$ can be mapped into $I$ using only powers of $g_i$ and $g_j$.
The next lemma finishes the proof.
\end{proof}

\begin{lemma}\label{lem:sinksource}
Suppose that $\alpha,\beta\in \Homeo^+(S^1)$ act with multi sink-source dynamics, and that the fixed point sets $\Fix(\alpha)$ and $\Fix(\beta)$ are disjoint. Let $p\in S^1$, and let $I\subset S^1$ be any open interval containing fixed points of  both $\alpha$ and $\beta$. Then $p$ can be sent into $I$ by an element of $\langle \alpha,\beta\rangle$.
\end{lemma}
\begin{proof}
Let $a_1$ be the first point of $\Fix(\alpha)$ encountered upon moving clockwise from $p$ toward $I$. If $a_1\in I$, a single element of $\langle \alpha\rangle$ suffices to move $p$ into $I$. Otherwise, by applying an element of $\langle\alpha\rangle$ and relabeling $p$, we can ensure $p$ is close enough to $a_1$ so that $[p,a_1]\cap \Fix(\beta)=\varnothing$. Now let $b_1$ be the first point of $\Fix(\beta)$ encountered upon moving clockwise from $a_1$ toward $I$. The same reasoning shows that an element of $\langle \beta\rangle$ suffices to either move $p$ into $I$, or move $p$ close enough to $b_1$ so that $[p, b_1]\cap \Fix(\alpha)=\varnothing$. The fact that $I$ contains fixed points of $\alpha$ and of $\beta$ ensures that we do not overshoot $I$. Since $\Fix(\alpha)$ and $\Fix(\beta)$ are finite, finitely many of these steps suffice to move $p$ into $I$.
\end{proof}

Note that if $\phi$ is $\R$-covered, many intervals in $\del\orb$ do not contain both endpoints of any leaf of $\orb^s$ or $\orb^u$. The next lemma shows that this in fact characterizes $\R$-covered flows.

\begin{corollary}\label{iseethearch}
Suppose that $\phi$ is non $\R$-covered, and let $I$ be an interval in $\del \orb$. Then there exists a leaf of $\orb^s$ with two ideal endpoints contained in $I$. The same is true for $\orb^u$. 
\end{corollary}

\begin{proof}
Let $x$ be a nonsingular, noncorner periodic point, fixed by some $g \in \pi_1(M)$.
\Cref{prop:minimal} implies that, up to translating $x$ and conjugating $g$, we may assume that some unstable ray $r$ through $x$ has its endpoint in $I$. Then for any regular stable leaf $\ell$ crossing $r$, the endpoints of $g^i(\ell)$ converge to $r(\infty) \in I$: otherwise the $g^i(\ell)$ would limit on a periodic leaf making a perfect fit with $r$, forcing $x$ to be a corner point by the  discussion in \Cref{sec:lozenges}. 
\end{proof}

\subsection{Rays in leaves and spike regions}
For the remainder of \Cref{sec:fibers} we fix a particular leaf $\lambda$ of $\wt \FF$.

We now recall some more fundamental work of Fenley, beginning with the relation between rays and geodesics in the leaf $\lambda$. 

\begin{lemma}[Rays vs. geodesics, Fenley]
\label{lem:omniFenley2}
Stable and unstable foliation rays in $\lambda$ enjoy the following properties:
\begin{enumerate}[label=(\roman*)]
\item \cite[Lemma 6.4]{Fen09}: 
There exists $\delta>0$ so that for any stable/unstable ray $r$ in $\lambda$, if $r^*$ denotes its geodesic representative, $r^*\subset N_\delta(r)$. Moreover, $\delta$ can be chosen \emph{independently} of $\lambda$.
\item \cite[Proposition 6.11]{Fen09}: If the leaf space of $\orb^s$ is Hausdorff, then there exists $k>0$, not depending on $\lambda$, such that all leaf slices of $\lambda^s$ are uniform $k$-quasigeodesics. The same is true for $\orb^u$.
\end{enumerate}
\end{lemma}

In general, leaf slices in $\lambda$ are {not} quasi-geodesics---in fact, there can  exist leaf slices whose ideal endpoints are equal. In analyzing this behavior, Fenley discovered formations that he calls spike regions, defined as follows.

A \emph{stable spike region} $\Sigma\subset\lambda$ is a closed, connected subset of $\lambda$ satisfying the following.
\begin{itemize}
\item $\Sigma$ is bounded by finitely many mutually disjoint stable leaf faces $\ell_1,\dots,\ell_n$, $n\ge 2$, each of which is regular to the side containing $\Sigma$.
\item Each endpoint of an $\ell_i$ in $\del \lambda$ is an endpoint of one of the other $\ell_i$'s.
\item $\Sigma$ contains no singular points in its interior. There is a distinguished point  $p\in \del\lambda$ such that if $\ell$ is a stable leaf in $\Sigma$, then both endpoints of $\ell$ are equal to $p$.
\end{itemize}

An \emph{unstable spike region} is defined similarly. See \Cref{fig:spike}.

\begin{figure}
\begin{center}
\includegraphics[height=1.5in]{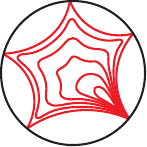}
\caption{A stable spike region in $\lambda$.}
\label{fig:spike}
\end{center}
\end{figure}

Thus a spike region $\Sigma$ has finitely many ends, each corresponding to a single ideal point in $\del \lambda$, and all the leaves in $\intr(\Sigma)$ nest down to a distinguished ideal point.

The following lemma of Fenley (\cite[Prop 6.7]{Fen09}) gives a useful characterization of spike regions and describes the behavior of their ends when projected to $M$.

\begin{lemma}[Spike regions, Fenley]
\label{lem:spike}
Let $\ell$ be a slice leaf of $\lambda^s$. Suppose that both ends of $\ell$ correspond to the same ideal point in $\del \lambda$. Then $\ell$ is contained in the interior of a stable spike region. Moreover, for each end of the spike region, the image in $M$ of each neighborhood of the end either is a Reeb annulus in a leaf of $\mc F$ or spirals on such a Reeb annulus.
\end{lemma}

We will need the following uniform finiteness statement for spike regions.

\begin{lemma}\label{finitespikes}
There are at most finitely many spike regions having a common endpoint $\sigma \in \del\lambda$. Moreover, the bound is independent of $\lambda$.
\end{lemma}

\begin{proof}
Let $\delta$ be the constant from \Cref{lem:omniFenley2} and fix $c>0$ to be the minimal width of any Reeb annulus in any leaf of $\mc F$ (as appearing in \Cref{lem:spike}), which exists by compactness of $M$. We will bound the number of spike regions with an ideal point at $\sigma$ by $\frac{2\delta}{c}+2$.

Suppose $\Sigma_1$ and $\Sigma_2$ are spike regions in $\lambda$ sharing the ideal point $\sigma$. Let $r_1$ and $r_2$ be two rays in the frontiers of $\Sigma_1$ and $\Sigma_2$, respectively, such that $r_1(\infty)=r_2(\infty)=\sigma$. By \Cref{lem:omniFenley2}(i), there exist sequences $(p_i)$ and $(q_i)$ in $r_1$ and $r_2$ respectively that escape compact sets and such that $d_\lambda(p_i,q_i)<2\delta$. In particular, for each $i$ we can find a path $\alpha^i$ from $p_i$ to $q_i$ whose length in $\lambda$ is at most $2\delta$. 

Denote the projection of $\alpha^i$ to $M$ by $\alpha^i_M$. These are paths of length at most $2\delta$, each contained in a leaf of $\mc F$. After passing to a subsequence they converge to a path $\alpha_M$ of length at most $2\delta$. By covering $\alpha_M$ with foliation charts, we can see that $\alpha_M$ is contained in a single leaf $\lambda'_M$ of $\mc F$. 
By \Cref{lem:spike}, the projections to $M$ of neighborhoods of the ends of $\Sigma_1$ and $\Sigma_2$ spiral around Reeb annuli $A_1$ and $A_2$. This implies that all accumulation points of the projections must lie in the leaves containing $A_1$ and $A_2$, so we conclude that $A_1$ and $A_2$ both lie in $\lambda'_M$. 
Hence, $\alpha_M$ is a path from $A_1$ to $A_2$.

If $\Sigma'$ is any spike region between $\Sigma_1$ and $\Sigma_2$, 
meaning that the rays of $\Sigma'$ limiting to $\sigma$ lie between $r_1$ and $r_2$, 
then for large $i$ the $\alpha^i$ cross the end of $\Sigma'$ associated to $\sigma$. Hence, as above, the limiting Reeb annulus associated to the projection of this end of $\Sigma'$ to $M$ 
is also crossed by $\alpha_M$. Since $\alpha_M$ has length at most $2\delta$, it crosses a Reeb annulus at most $2\delta/c$ times, where ``crosses" refers to passing from one boundary component to the other. This means that the number of spike regions between $\Sigma_1$ and $\Sigma_2$ is at most $2\delta/c$. Since $\Sigma_1$ and $\Sigma_2$ were arbitrary, we see the number of spike regions having an ideal point at $\sigma$ is at most $\frac{2\delta}{c}+2$ as desired.
\end{proof}

Finiteness of spike regions implies the following `non-collapsing' lemma.

\begin{lemma}\label{lem:nointervalcollapse}
There exists no open interval in $\ol \Omega_\lambda\cap \del\orb$ that is mapped to a single point by $\pi_\lambda$.
\end{lemma}

\begin{proof}
Suppose to the contrary that such an interval $I$ exists, and let $\sigma$ be the image of $I$ under $\pi_\lambda$.

If $\phi$ is not an $\R$-covered Anosov flow, then repeated applications of \Cref{iseethearch} produce arbitrarily many spike regions with an ideal point at $\sigma$, contradicting \Cref{finitespikes}.

Otherwise, $\phi$ is an $\R$-covered Anosov flow. Note that since $M$ is compact, $\phi$ must make a definite positive angle with $\FF$ that is bounded below by some constant. This implies that at each point in $\lambda$, the stable and unstable leaves through that point make a positive angle that is bounded below by some constant. This in turn gives that there exists $\epsilon>0$ such that for any leaf $\ell$ of $\lambda^s$, the restriction of $\lambda^u$ to $N_\epsilon(\ell)$ is a standard foliation by intervals.

By \Cref{lem:omniFenley2}(ii), the leaves of the foliations $\lambda^u$ and $\lambda^s$ are uniform quasigeodesics. Hence there exists $K$ such that any two rays in $\lambda$ with the same ideal endpoint eventually lie in each other's $K$-neighborhoods.

Fix some $n\in \mathbb N$, and let $\rho_1,\rho_2,\dots, \rho_n$ be disjoint rays in $\orb$ with pairwise distinct endpoints in $I$, so that $\rho_i$ is stable for $i$ odd and unstable for $i$ even. This choice is possible because $\phi$ is skew Anosov. 
We let $r_i$ be the endpoint of $\rho_i$, and label the $\rho_i$ so that $r_{i+1}\in [r_i, r_{i+2}]$ for $1\le i\le n-2$. Let $\eta_i$ be the lift of $\rho_i$ to $\lambda$, so that $\eta_i(\infty)=\sigma$.
By the pigeonhole principle, some $\eta_i$ eventually intersects the $\frac{K}{n-1}$-neighborhood of $\eta_{i+1}$. 
In this argument, we are free to choose $n$ large enough so that $\frac{K}{n-1}<\epsilon$. Since the $\eta_i$'s are disjoint, this contradicts that the foliation of each $N_\epsilon(\eta_i)$ by leaves of the other foliation is standard.
\end{proof}

\subsection{Coherence and proximity}

Let $\lambda$ be a leaf of $\wt\FF$. We say that a deck transformation $g$ of $\wt M$ is \emph{ordering} with respect to $\lambda$ if $\lambda$ separates $g^{-1}\lambda$ and $g\lambda$ in $\wt M$. In this situation, $g$ is either \emph{positive} or \emph{negative}, depending on whether $g\lambda$ lies to the positive or negative side of $\lambda$, respectively. 
Let $\gamma$ be a flowline in $\wt M$ and let $\rm{stab}_0(\gamma)$ denote the finite index subgroup of its stabilizer that fixes the stable and unstable leaves through $\gamma$.

The non-ordering situation can be understood as follows: suppose that $g\lambda\ne\lambda$ and $g$ is not ordering for $\lambda$.  If $g\lambda$ and $g^{-1}\lambda$ lie above $\lambda$, then it follows that for any element $g^{i} \lambda\in\langle g\rangle\cdot \lambda$, all other elements lie above it. If $g\lambda$ and $g^{-1}\lambda$ lie below $\lambda$, the symmetric statement holds.

We say that $\gamma$ is \emph{coherent} with respect to $\lambda$ if some, and hence every, element $g\in \rm{stab}_0(\gamma)$ is ordering for $\lambda$ and translates $\gamma$ in the same direction as it translates $\lambda$. If $g$ is chosen to translate $\gamma$ in the forward direction, this means that $\gamma$ is coherent if and only if $g$ is ordering for $\lambda$ and $g \lambda$ lies on the positive side of $\lambda$.
We use the same terminology for periodic points in $\orb$, according to the properties of their associated orbits.

From our discussion in \Cref{sec:lozenges}, the following lemma is immediate:

\begin{lemma}\label{lem:coh}
If one corner of a lozenge is ordering for $\lambda$ then both are, and one corner is coherent with respect to $\lambda$ while the other is not. 
\end{lemma}

If $p$ is in $\Omega_\lambda$ and fixed by $g$, then $g$ is coherent with respect to $\lambda$ since the corresponding orbit $\gamma_p$ intersects $\lambda$. We also have:

\begin{lemma}
If $p$ is a point of $\fr \Omega_\lambda$ and $g\in \stab_0(p)$, then either $g$ is ordering with respect to $\lambda$ or $g \lambda=\lambda$. 
\end{lemma}

\begin{proof}
Suppose that $g \lambda\ne\lambda$ and $g$ is not ordering with respect to $\lambda$. Let $\gamma$ be the flowline of $\wt\phi$ projecting to $p$. Up to replacing $g$ by $g^{-1}$, we can assume that $g$ translates $\gamma$ in the forward direction.

Assume without loss of generality that the component of $\fr\Omega_\lambda$ containing $p$ is stable. Up to replacing $p$ by another periodic point in the same component of $\fr \Omega_\lambda$, we can assume there is a half-leaf $\ell^u$ of $\orb^u$ emanating from $p$ into $\Omega_\lambda$.

Let $L^u$ be the half-leaf of $\wt W^u$ projecting to $\ell^u$.
Let $\alpha$ be a flowline in $L^u$ that passes through $\lambda$, and let $a$ be its projection to $\orb$. Then $a$ and $ga$ lie in $\ell^u$, and $ga$ is closer to $p$ than $a$ is. Since $g$ is not ordering for $\lambda$, $g\alpha$ does not intersect $\lambda$, so $ga\notin\Omega_\lambda$, a contradiction.
\end{proof}

If $\ell$ is a periodic leaf slice in $\orb$, then we say that $\ell$ is coherent with respect to $\lambda$ when its periodic points are coherent (there will be more than one only if $\ell$ contains one or more blown segments). 

\begin{lemma}
\label{lem:coherent}
Let $p\in \orb-\Omega_\lambda$. Let $g \in \stab_0(p)$ and suppose that $g$ is ordering for $\lambda$. 
Let $\gamma$ be the flowline lying over $p$ in $\wt M$, and suppose that a half-leaf $H$ of $\wt W^u$ or $\wt W^s$ containing $\gamma$ intersects $\lambda$. 
Denote the intersection $H\cap \lambda$ by $\ell$. 
\begin{enumerate}[label=(\roman*)]
\item \label{item:coherent} If $p$ is coherent for $\lambda$, and $\ell$ is oriented so that its projection $\Theta(\ell)$ is oriented toward $p$, then $\ell$ eventually lies a bounded distance from $\gamma$.

\item \label{item:notcoherent}If $p$ is not coherent for $\lambda$, and $\ell$ is oriented so that its projection $\Theta(\ell)$ is oriented \emph{away from} $p$, then $\ell$ eventually lies a bounded distance from $\gamma$.
\end{enumerate}
\end{lemma}

\begin{proof}
In this proof we continue to use a subscript $M$ to denote the projection of an object to $M$.

We assume that the orbits in $H$ are asymptotic to $\gamma$ in the backward direction. Let $q$ be a point in $\ell=H\cap \lambda$.
Let $\sigma_0$ be an embedded path in $W^u(\gamma_M)$ transverse to $\phi$ in $W^u(\gamma_M)$ and connecting $q_M$ to $\gamma_M$. Let $\sigma$ be the lift of $\sigma_0$ to $\wt M$ that starts at $q$. 

First suppose that $p$ is coherent with respect to $\lambda$. Then $g\lambda$ lies above $\lambda$, and the backward flowline from $gq$ must intersect $\lambda$. This forces $\ell$ to intersect $g\sigma$ at a point closer to $\gamma$ than $gq$. 
See the lefthand side of \Cref{fig:proximity}.
In $M$, this forces $\ell_M=\lambda_M\cap H_M$ to spiral on a closed curve parallel to $\gamma$. If we orient $\ell$ compatibly with the direction of this forced spiraling, its projection to $\orb$ will be oriented toward $p$. This completes the proof in this case.

\begin{figure}
\begin{center}
\includegraphics[height=1.75in]{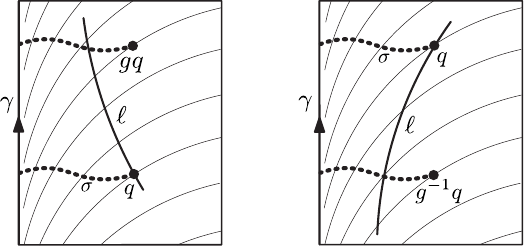}
\caption{Possible pictures in the half leaf $H$ from the proof of \Cref{lem:coherent}. Left: the case where $\gamma$ is coherent with respect to $\lambda$. Right: the case where $\gamma$ is not coherent with respect to $\lambda$.}
\label{fig:proximity}
\end{center}
\end{figure}

Now suppose that $p$ is not coherent with respect to $\lambda$. Then the backward flowline from $g^{-1}q$ intersects $\lambda$, which forces $\ell$ to intersect $g^{-1}\sigma$
at a point closer to $\gamma$ than $g^{-1}q$. See the righthand side of \Cref{fig:proximity}. In $M$, this again forces $\ell_M$ to spiral on a closed curve. In contrast with the previous case, the direction of the spiraling is such that if we orient $\ell$ accordingly, its projection to $\orb$ will be oriented away from $p$.

The case where orbits in $H$ are asymptotic to $\gamma$ in the forward direction is similar, although $\ell$ will eventually fellow travel $\gamma$ in the opposite direction (in $\wt M$) in both the coherent and not coherent cases.
\end{proof}

\subsection{Endgame: gaps are spanned by finite chains}

The \emph{span} of a frontier component $\ell$ of $\Omega_\lambda$ is the component of $\del\orb-\del \ell$ not containing any ideal points of $\Omega_\lambda$, where $\del \ell$ is the set of $\ell$'s ideal endpoints.

A \emph{frontier chain} for the shadow $\Omega_\lambda$ of a leaf $\lambda$ of $\wt \FF$ is a sequence $\ell_1,\ell_2\dots, \ell_n$ of components of $\fr(\Omega_\lambda)$ such that any two consecutive components share an ideal point. The \emph{span} of a the chain $\{\ell_1,\dots,\ell_n\}$ is the smallest open interval containing the spans of all $\ell_i$. Alternatively, this is the interior of the closed interval obtained by taking the union of the closures of each individual span. 
The \emph{length} of a frontier chain is its number of components.

Finally, a frontier chain is a \emph{limit chain} if there is a sequence of leaves in $\orb$ intersecting $\Omega_\lambda$ limiting on the entirety of the chain and nothing else. The importance of limit chains lies in the following observation: 
if the closure of the span of a limit chain is contained in a gap of $\pi_\lambda$, then leaves in $\Omega_\lambda$ limiting to the chain eventually determine leaves in $\lambda$ with the same endpoints in $\partial \lambda$, and hence a spike region in $\lambda$ by \Cref{lem:spike}. 

\begin{lemma}
\label{lem:bounded_chain}
Any frontier chain has finite length. Moreover, the bound is independent of $\lambda$. 
\end{lemma}
\begin{proof}
Consider a hypothetical frontier chain of infinite length.

Suppose there are infinitely many components of the chain contained in branching leaves or leaves containing blowup segments. Since they lie in the same chain, all these components are fixed by some common $g\in \pi_1(M)$, contradicting \Cref{zplusz}.

Otherwise infinitely many frontier components in the chain are not contained in branching leaves and do not meet blowup segments. Then each of these components is accumulated upon by leaves from the $\Omega_\lambda$ side, i.e. by leaves that intersect $\Omega_\lambda$. In particular, each component is itself a limit subchain.
This gives rise to infinitely many spike regions in $\lambda$ sharing an ideal point, contradicting \Cref{finitespikes}.

Independence of the bound from $\lambda$ follows from the uniformity statements in \Cref{zplusz} and \Cref{finitespikes}.
\end{proof}

The next lemma is our primary tool in the absence of limit subchains:

\begin{lemma}
\label{findgoodray}
Consider a frontier chain $C$ of $\Omega_\lambda$. 

Suppose that $C$ has no subchain which is a limit chain. Then there exists a foliation ray $r$ in $\Omega_\lambda$, terminating at a point in $C$ or an ideal point of $C$, such that its lift $\wt r$ to $\lambda$ stays bounded distance from any orbit corresponding to a periodic point in $C$.
\end{lemma}

\begin{proof}
Under the hypotheses of the lemma, $C$ must contain branching leaves or meet a blown segment intersecting $\Omega_\lambda$. Hence all elements in $C$ are fixed by a common $g\in\pi_1(M)$.

By \Cref{lem:coherent} \cref{item:coherent}, it suffices to consider the case when all periodic points in $C$ are incoherent with respect to $\lambda$. It follows from \Cref{lem:coh} that the components of $C$ are either all stable or all unstable. For concreteness we consider the case when all are stable.

Since no subchains are limit chains, we can find a stable ray $r$ in $\Omega_\lambda$ either contained in a blowup segment and terminating at a periodic point in $C$, or terminating at an ideal point of $C$. Let $\ell$ be the half leaf of $\orb^s$ containing $r$. Every periodic point of $\ell$ is incoherent with respect to $\lambda$ (see \Cref{lem:coh}). Hence by \Cref{lem:coherent} \cref{item:notcoherent}, the lift of $r$ to $\lambda$ stays a bounded distance from the flowlines corresponding to periodic points in $\ell$. 

Since these are connected to the periodic points in $C$ by chains of lozenges, these flowlines project to orbits in $M$ which are homotopic or antihomotopic. The lemma follows.
\end{proof}

The following is a restatement of \Cref{th:gaps}. 
\begin{theorem}\label{thm:gaps}
Every gap of $\pi_\lambda$ is spanned by a finite chain of frontier leaves. 
\end{theorem}
Note that by \Cref{lem:bounded_chain}, the length of such a chain is uniformly bounded.
\begin{proof}
Suppose by way of contradiction that there is a gap $G$ of $\pi_\lambda$ that is not spanned by a finite chain as in the statement. By \Cref{lem:nointervalcollapse}, there are infinitely many maximal frontier chains with both endpoints in $G$. Each of these chains either has a subchain which is a limit chain, or does not.

Suppose there are infinitely many chains with limit subchains. As before, each limit subchain whose span has closure contained in a single gap gives rise to a spike region of $\lambda$. 
Hence we can find arbitrarily many spike regions in $\lambda$ with the same ideal point, contradicting \Cref{finitespikes}.

Otherwise, there are infinitely many of these frontier chains having no limit chain as a subchain. Note that all these leaves are branching or contain blown segments. 
By \Cref{findgoodray}, for each chain we can find a foliation ray in $\lambda$ (coming from one in $\Omega_\lambda$) staying a bounded distance from the periodic flowlines corresponding to that chain.

All these rays in $\lambda$ have the same ideal point because they are associated to chains in our fixed gap $G$. By \Cref{lem:omniFenley2}, any $2$ of these rays have sequences of points escaping to their positive ends that stay bounded distance from each other. This means the flowlines have the same property, up to enlarging the bound. These flowlines are quasigeodesics, because they are lifts of closed curves. This implies that the projections of these flowlines is a family of pairwise homotopic or antihomotopic periodic orbits.

We conclude that all the non-limit chains are fixed by a common group element corresponding to the above homotopy class, contradicting \Cref{zplusz}. 
\end{proof}

\section{Minimal universal circles}\label{sec:minimal} \label{sec:skewnonregulating}

In this section we investigate the minimality of our universal circles.

\medskip

Given a universal circle, it is possible to obtain a new universal circle by performing a \emph{Denjoy blowup} on the orbit of a point; see \cite[\S 5.1]{Calegari_promoting}. This operation replaces every point in the orbit by a closed interval, and the resulting universal circle contains intervals which are collapsed by every leafwise monotone map. A universal circle $(\univ, \{m_\lambda\}, \rho)$ is \emph{minimal} if no open interval in $\univ$ is contained in a gap of each $m_\lambda$. Calegari showed that every non-minimal universal circle can be obtained from Denjoy blowups on a minimal universal circle \cite[Lemma 5.1.2]{Calegari_promoting}. 

On the other hand, the word minimal is also applied to group actions. These homonyms are related in a basic way:

\begin{lemma}\label{ifminthenmin}
Let $\mathscr C=(\univ, \{m_\lambda\}, \rho)$ be a universal circle for $\FF$. If the associated action $\pi_1(M) \curvearrowright \univ$ is minimal, then $\mathscr C$ is a minimal universal circle.
\end{lemma}

\begin{proof}
Let $I\subset S^1$ be an interval, and let $G$ be a gap of $m_\lambda$ for some leaf $\lambda$ of $\wt\FF$. Let $p$ be an endpoint of $G$. Choose $g\in \pi_1(M)$ so that $g\cdot p$ lies interior to $I$. Note that $g\cdot p$ is the endpoint of a gap of $\pi_{g\cdot \lambda}$, so $I$ is not contained in a gap of $\pi_{g\cdot \lambda}$.
\end{proof}

Hence, the combination of \Cref{simUC}, \Cref{prop:minimal}, and \Cref{ifminthenmin} imply that universal circles associates to non $\R$-covered flows are minimal.

\smallskip

We say that $\phi$ is \emph{regulating} for $\FF$ if each flowline of $\wt \phi$ intersects every leaf of $\wt \FF$. Hence $\phi$ is regulating for $\FF$ exactly when the shadow of any leaf of $\wt \FF$ is equal to $\orb$.
Note that if $\phi$ is regulating for $\FF$, then $\mathscr C_\phi=(\del \orb, \rho_\phi, \{\pi_\lambda\})$ is minimal, because no interval of $\del\orb$ is contained in a gap of \emph{any} $\pi_\lambda$ by \Cref{lem:nointervalcollapse}.

In light of this, it only remains to consider the case where $\phi$ is skew Anosov and not regulating for $\FF$.

Fix an identification of $\orb$ with $\{(x,y)\mid |y-x|< 1\}$ so that $\orb^s$ consists of vertical lines and $\orb^u$ consists of horizontal lines, let $\del_+\orb$ be the ``top" of the flow space, i.e. the line $y=x+1$, and let $\del_-\orb$ be the line $y=x-1$.

In \cite[Section 7]{fenley2005regulating}, Fenley shows that this situation is quite rigid. In fact, up to collapsing any trivial pockets of parallel leaves, $\FF$ is topologically conjugate to the stable or unstable foliation of $\phi$. More relevant to our current discussion is that every shadow $\Omega_\lambda$ of a leaf $\lambda$ of $\wt \FF$ is a ``triangle": it has two frontier components, a stable leaf and an unstable leaf making a perfect fit. Moreover, the ideal points corresponding to these perfect fits either all lie in $\del_+\orb$ or all lie in $\del_-\orb$ (see \cite[Prop. 7.7]{fenley2005regulating}). In particular, either $\del_+\orb$ or $\del_-\orb$ is collapsed to a point by each $\pi_\lambda$. This gives:

\begin{proposition}\label{prop:when_min}
Under the conditions of \Cref{maintheorem}, 
the universal circle $\mathscr C_\phi=(\del \orb, \rho_\phi, \{\pi_\lambda\})$ is minimal if and only if the following conditions are not all satisfied:
\begin{itemize}
\item $\varphi$ is skew Anosov,
\item $\FF$ is $\R$-covered,
\item $\varphi$ is nonregulating for $\FF$. 
\end{itemize}
\end{proposition}

\section{Depth one foliations}

In this final section we prove a rigidity result, \Cref{thm:depth1rigidity}.

\subsection{Boundary dynamics of first return maps}

Before we begin, we need to rule out a certain structure in the flow space $\orb$. A \emph{$g$-sail of size $n$} is a disk $D$ in $\orb$ bounded by $r^s, r^u, \ell_1,\dots, \ell_n$ where: 
\begin{itemize}
\item$r^u$ and $r^s$ are adjacent stable and unstable leaf rays starting at a point $p\in \orb$ called the \emph{tip} of the sail,
\item $\ell_1,\dots, \ell_n$ are leaf slices such that $\ell_1$ shares an ideal point with $r^s$, $\ell_n$ shares an ideal point with $r^u$, and $r_i$ shares an ideal point with each of $r_{i-1}$ and $r_{i+1}$ for $1<i<n$, and
\item all of $r^s, r^u, \ell_1,\dots, \ell_n$ are fixed by $g$.
\end{itemize} 
See \Cref{fig:sail}.

\begin{figure}
\begin{center}
\includegraphics[]{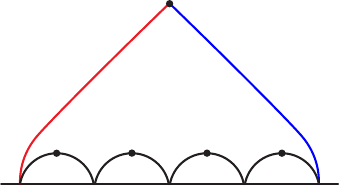}
\caption{A $g$-sail of size 4, where the indicated points are fixed by a common group element $g$. The black leaf slices are allowed to be stable or unstable.}
\label{fig:sail}
\end{center}
\end{figure}

\begin{lemma}\label{nosails}
Let $\phi$ be a pseudo-Anosov flow. Then its flow space $\orb$ has no $g$-sails.
\end{lemma}

\begin{proof}
We first observe that there can be no $g$-sail of size 0, because no two half leaves based at the same point in $\orb$ terminate at the same point in $\del\orb$.

Proceeding by induction, suppose there are no $g$-sails of size less than $n$, for any $g\in \pi_1(M)$, and suppose that $A$ is a $g$-sail of size $n$ bounded by $r^s, r^u, \ell_1,\dots, \ell_n$. The tip of $A$ must be the corner point of a lozenge $L$ contained in $A$. Let $q$ be the other corner point of this lozenge. Then it is either the case that $q$ lies in $\ell_1$, $\ell_n$, or $\intr(A)$. If $q\in \ell_1$, then $A-L$ is a $g$-sail of size $n-1$, a contradiction. If $q\in \intr(A)$, then $q$ is the tip of $\ge3$ $g$-sails contained in $A$ (3 if $q$ is a regular point, more if $q$ is singular). All of these have smaller size than $A$ since the half leaves from $q$ must end at separate ideal points of $A$, so this is also a contradiction.
\end{proof}

We will also need the following proposition.

\begin{proposition}
\label{prop:return}
Let $\varphi$ be an almost pseudo-Anosov flow on $M$ and let $\mc F$ be a foliation transverse to $\phi$. Let $\lambda_M$ be a leaf of $\mc F$ with a well-defined first return map $f \colon \lambda_M \to \lambda_M$. 

Let $\lambda$ be the universal cover of $\lambda_M$. Then for any fixed point $\wt p$ of a lift $\wt f$ of a power of $f$ to $\lambda$, the half-leaves based at $\wt  p$ determine distinct endpoints on $\partial\lambda$. 

Moreover, the induced action of $\wt f$ on $\partial \lambda$ has multi sink-source dynamics.
\end{proposition}

\begin{proof}
Fix an identification of $\lambda$ with a leaf of $\wt \FF$ that projects to $\lambda_M$ and consider the deck transformation $g$ such that flowing $g \lambda$ upward to $\lambda$ induces $\wt f$.
Letting $p$ denote the projection of $\wt p$ to $\Omega_\lambda\subset \orb$, we have that $g  p =  p$ and $g$ fixes $\Omega_\lambda$. 

Up to replacing $g$ by a finite power, we can assume that $g\in \Stab_0(p)$ (recall this means $g$ fixes the half leaves at $p$). If the endpoints of these half-leaves are not distinct {in} $\partial \lambda$, then there is an adjacent pair of foliation rays $r^u$ and $r^s$ at $p$ so that $\pi_\lambda(r^u(\infty)) = \pi_\lambda(r^s(\infty))$. This implies that these endpoints are contained in a single gap of $\pi_\lambda$. By the characterization of gaps in \Cref{thm:gaps}, there is a chain of frontier components $\ell_1, \ldots, \ell_n$ such that
\begin{itemize}
\item either $\ell_1$ crosses $r^s$ or has a common endpoint, and 
\item either $\ell_n$ crosses $r^u$ or has a common endpoint. 
\end{itemize}

Since $g$ fixes $\Omega_\lambda$, it acts on its set of frontier leaves. In particular, it must fix the leaves $\ell_1, \ldots, \ell_n$, 
since it fixes $r^u(\infty)$ and $r^s(\infty)$. We conclude that if $r^s$ crosses $\ell_1$ then $r^s\cap \ol\Omega_\lambda$ is a blown segment. Otherwise, $r^s$ shares an ideal endpoint with $\ell_1$. The analogous statement holds for $r^u$ and $\ell_n$.
Collapsing blown segments produces a $g$-sail, contradicting \Cref{nosails}.

\smallskip

We now prove that $\wt f$ acts on $\del \lambda$ with multi sink-source dynamics.

Let $\eta \in \partial \lambda$ be any fixed point of $\wt f$. Again applying \Cref{thm:gaps}, if the preimage $\pi_\lambda^{-1}(\eta)$ is not a single point, then it is spanned by a finite chain of frontier leaves. This chain is fixed by $g$, so each leaf in the chain is fixed by $g$ also. 
The endpoints of these leaves in $\partial \orb$ gives a finite set of isolated fixed points of $g$, each of which is either a source or a sink in $\partial \orb$ by \Cref{sink}. 
Since the map $\pi_\lambda\colon \partial \orb \to \partial \lambda$ is monotone, to show that $\eta$ is a source or sink in $\partial \lambda$, it suffices to show that the endpoints of the gap 
$\pi_\lambda^{-1}(\eta)$ are either both sources or both sinks. 
But this is clear because the endpoints of any leaf fixed by $g$ whose endpoints are isolated fixed points are either both sources or both sinks, so the same is true for a finite chain.

It remains to consider the case where $\pi_\lambda^{-1}(\eta)$ is a single point $x\in \partial \orb$, necessarily fixed by $g$. Again applying \Cref{sink}, if $x$ is not a limit of points in $\partial \orb$ fixed by $g$, then it is a sink or a source. As above, we conclude the same for $\eta \in \partial \lambda$.

It therefore suffices to show that any fixed point of $g$ in $\partial \orb$ which is a limit of other fixed points is contained in an open interval spanned by a frontier component of $\Omega_\lambda$. 
If $x$ is such a fixed point, then as in the proof of \Cref{sink} we see that $x$ is the accumulation point of a chain of lozenges $(L_i)_{i\ge0}$ such that $p$ is a corner point of $L_0$ and for large $i$ the lozenges intersect only in their corners (see \Cref{fig:lozengechain}).
In particular for $i$ sufficiently large there is a leaf slice 
$\ell$ containing a side of $L_i$ that is disjoint from $\Omega_\lambda$ and separates $p$ from $x$. Since $\Omega_\lambda$ can intersect at most one corner of a given periodic lozenge by \Cref{lem:coh},
we can find a frontier component of $\Omega_\lambda$ separating $p$ from $x$. We conclude that $x$ is contained the span of this frontier component.
\end{proof}

\subsection{Depth one rigidity}

A foliation $\FF$ of $M$ is \emph{depth one} if every noncompact leaf accumulates only on compact leaves. By collapsing pockets of parallel leaves, we can always assume that a depth one foliation has only finitely many compact leaves. If $\FF_0$ denotes the union of the compact leaves, then each component of $M-\FF_0$ will fiber over the circle with fibers the leaves of $\FF$.

\begin{proposition}\label{onlyperfectfits}
Let $\mc F$ be a depth one foliation transverse to almost pseudo-Anosov flows $\varphi$ and $\psi$. If $\varphi$ has no perfect fits, then neither does $\psi$. 
\end{proposition}

\begin{proof}
Suppose that $\varphi$ has no perfect fits.

Let $N_1, \ldots, N_k$ be the components of $M$ cut along $\FF_0$. Our proof will using some standard facts about the monodromies of depth one foliations which can be found in \cite{landry2023endperiodic}.

\begin{claim}
Every closed orbit of $\psi$ in $N_i$ is homotopic in $N_i$ to a closed orbit of $\varphi$. 
\end{claim}

To prove the claim, set $N = N_i$, let $L$ be any depth one leaf of $N$, and $f_{\phi/\psi} \colon L \to L$ be the associated first return map under $\phi/\psi$. Then $f_\phi$ and $f_\psi$ are homotopic. 
The claim is equivalent to the statement that every periodic point of $f_\psi$ is Nielsen equivalent to a periodic point of $f_\phi$. 

Suppose that $\wt p$ is a fixed point of some lift of some power of $f_\psi$, which we simply denote by $\wt f_\psi$. 
Let $\wt f_\phi$ be the corresponding lift of the same power of $f_\phi$, obtained by lifting a homotopy.  Then the actions of $\wt f_\phi$ and $\wt f_\psi$ on $\partial \wt L$ agree by \cite[Corollary 4.2]{CaCo13}.

By \Cref{prop:return}, $\wt f_\psi$ has at least $4$ fixed points on $\partial \wt L$ and acts with multi sink-source dynamics, so $\wt f_\phi$ does too. According to \cite[Theorem 4.1]{landry2023endperiodic}, $\wt f_\phi$ must also fix a point $\wt q$ in $\wt L$. (Technically, the statements there are for return maps to pseudo-Anosov \emph{suspension} flows, but the proofs go through without change when the flow has no perfect fits.) This establishes the required Nielsen equivalence and completes the proof of the claim.

\smallskip

Returning to the proof, suppose for a contradiction that $\psi$ has perfect fits. By results of Fenley, this is true if and only if $\psi$ has anti-homotopic orbits $\gamma_1$ and $\gamma_2$, i.e. periodic orbits such that $\gamma_1$ is homotopic to the inverse of $\gamma_2$ (\cite[Theorem B]{Fen16}, \cite[Theorem 4.8]{fenley1999foliations}). 
It is clear that neither of these orbits can intersect $\FF_0$, since each component of $\FF_0$ is dual to a cohomology class that is nonnegative on closed orbits. Hence, $\gamma_1 \subset N_i$ and $\gamma_2 \subset N_j$. (It is possible that $i=j$, but they cannot be anti-homotopic \emph{within} $N_i$ since each is positively transverse to an associated depth one leaf $L_i$.)

By the claim, each of these orbits is homotopic to a corresponding orbit of $\varphi$ by a homotopy supported in its associated component of $M \ssm \FF_0$. We conclude that $\varphi$ must also have anti-homotopic orbits and thus perfect fits, a contradiction.
\end{proof}

Applying Huang's result that a depth one foliation admits a unique transverse pseudo-Anosov flow without perfect fits (\cite[Corollary 1.3]{Huang_UC}), we conclude: 

\begin{theorem}\label{thm:depth1rigidity}
Suppose that the depth one foliation $\mc F$ is transverse to a pseudo-Anosov flow $\varphi$ without perfect fits. Then $\varphi$ is the unique pseudo-Anosov flow transverse to $\FF$ up to orbit equivalence.
\end{theorem}

We expect \Cref{thm:depth1rigidity} holds with \emph{almost} transversality replacing transversality. Since \Cref{onlyperfectfits} is proved in that level of generality, it would suffice to generalize the result of Huang (\cite{Huang_UC}) to almost pseudo-Anosov flows transverse to depth one foliations.

\subsection*{Acknowledgements} 
We thank Sérgio Fenley for the many important results he has contributed to this area. We also thank Thomas Barthelmé, Junzhi Huang, and Chi Cheuk Tsang for helpful comments.

\bibliography{universal_circles.bib}

\def\cprime{$'$} \def\cprime{$'$}
\providecommand{\bysame}{\leavevmode\hbox to3em{\hrulefill}\thinspace}
\providecommand{\MR}{\relax\ifhmode\unskip\space\fi MR }
\providecommand{\MRhref}[2]{%
  \href{http://www.ams.org/mathscinet-getitem?mr=#1}{#2}
}
\providecommand{\href}[2]{#2}
\begin{thebibliography}{BBM24}

\bibitem[AT24]{AgolTsang}
Ian Agol and Chi~Cheuk Tsang, \emph{Dynamics of veering triangulations:
  infinitesimal components of their flow graphs and applications}, Algebraic
  and Geometric Topology \textbf{24} (2024), no.~6, 3401--3453.

\bibitem[Bar95]{barbot1995caracterisation}
Thierry Barbot, \emph{Caract{\'e}risation des flots d'anosov en dimension 3 par
  leurs feuilletages faibles}, Ergodic Theory and Dynamical Systems \textbf{15}
  (1995), no.~2, 247--270.

\bibitem[BBM24]{barthelme2024non}
Thomas Barthelm{\'e}, Christian Bonatti, and Kathryn Mann, \emph{Non-transitive
  pseudo-anosov flows}, arXiv preprint arXiv:2411.03586 (2024).

\bibitem[BFH16]{bucher_semiconj}
Michelle Bucher, Roberto Frigerio, and Tobias Hartnick, \emph{A note on
  semi-conjugacy for circle actions}, L'Enseignement Math{\'e}matique
  \textbf{2} (2016), no.~62, 317--360.

\bibitem[BFM22]{BFM}
Thomas Barthelm\'e, Steven Frankel, and Kathryn Mann, \emph{Orbit equivalences
  of pseudo-{A}nosov flows}, Preprint at arXiv:2211.10505, 2022.

\bibitem[BGH23]{boyer2023recalibrating}
Steven Boyer, C~Gordon, and Ying Hu, \emph{Recalibrating {R}-order trees and
  {H}omeo(s1)-representations of link groups}, arXiv preprint arXiv:2306.10357
  (2023).

\bibitem[Bon23]{bonatti2301action}
Christian Bonatti, \emph{Action on the circle at infinity of foliations of
  $\mathbb{R}^2$}, arXiv preprint arXiv:2301.04530, 2023.

\bibitem[Cal00]{Calegari2000Rcovered}
Danny Calegari, \emph{The geometry of $r$-covered foliations}, Geometry and
  Topology \textbf{4} (2000), 457--515.

\bibitem[Cal06a]{Calegari_promoting}
\bysame, \emph{Promoting essential laminations}, Inventiones mathematicae
  \textbf{166} (2006), 583--643.

\bibitem[Cal06b]{calegari2006universal}
\bysame, \emph{Universal circles for quasigeodesic flows}, Geometry \& Topology
  \textbf{10} (2006), no.~4, 2271--2298.

\bibitem[Cal07]{2007CalegariBook}
Danny Calegari, \emph{Foliations and the geometry of 3-manifolds}, Oxford
  Mathematical Monographs, Oxford University Press, Oxford, United Kingdom,
  2007.

\bibitem[Can93]{Candel_uniformization}
Alberto Candel, \emph{Uniformization of surface laminations}, Annales
  Scientifiques de l'\'Ecole Normale Sup\'erieure \textbf{26} (1993), no.~4,
  489--516.

\bibitem[CC13]{CaCo13}
John Cantwell and Lawrence Conlon, \emph{Hyperbolic geometry and homotopic
  homeomorphisms of surfaces}, Geometriae Dedicata \textbf{177} (2013), 27--42.

\bibitem[CD03]{CalDun_UC}
Danny Calegari and Nathan Dunfield, \emph{Laminations and groups of
  homeomorphisms of the circle}, Inventiones mathematicae \textbf{152} (2003),
  149--204.

\bibitem[CL24]{calegari2024zippers}
Danny Calegari and Ino Loukidou, \emph{Zippers}, arXiv preprint
  arXiv:2411.15610 (2024).

\bibitem[Fen94]{fenley1994anosov}
S{\'e}rgio~R Fenley, \emph{Anosov flows in 3-manifolds}, Annals of Mathematics
  \textbf{139} (1994), no.~1, 79--115.

\bibitem[Fen95]{fenley1995quasigeodesic}
\bysame, \emph{Quasigeodesic anosov flows and homotopic properties of flow
  lines}, Journal of Differential Geometry \textbf{41} (1995), no.~2, 479--514.

\bibitem[Fen99]{fenley1999foliations}
S{\'e}rgio~R. Fenley, \emph{Foliations with good geometry}, J. Amer. Math. Soc.
  \textbf{12} (1999), no.~3, 619--676.

\bibitem[Fen02]{Fenley2002Rcovered}
S{\'e}rgio Fenley, \emph{Foliations, topology and geometry of 3-manifolds:
  $\mathbb r$-covered foliations and transverse pseudo-{A}nosov flows},
  Commentarii Mathematici Helvetici \textbf{77} (2002), 415--490.

\bibitem[Fen05]{fenley2005regulating}
S{\'e}rgio~R. Fenley, \emph{Regulating flows, topology of foliations and
  rigidity}, Transactions of the American Mathematical Society \textbf{357}
  (2005), no.~12, 4957--5000.

\bibitem[Fen09]{Fen09}
\bysame, \emph{Geometry of foliations and flows {I}: {A}lmost transverse
  pseudo-{A}nosov flows and asymptotic behavior of foliations}, Journal of
  Differential Geometry \textbf{81} (2009), 1--89.

\bibitem[Fen12]{fenley2012ideal}
S{\'e}rgio Fenley, \emph{Ideal boundaries of pseudo-{A}nosov flows and uniform
  convergence groups with connections and applications to large scale
  geometry}, Geom. Topol. \textbf{16} (2012), no.~1, 1--110.

\bibitem[Fen13]{Fenley2013rigidity}
\bysame, \emph{Rigidity of pseudo-{A}nosov flows transverse to
  $\mathbb{R}$-covered foliations}, Commentarii Mathematici Helvetici
  \textbf{88} (2013), 643--676.

\bibitem[Fen16]{Fen16}
S{\'e}rgio~R. Fenley, \emph{Quasigeodesic pseudo-{A}nosov flows in hyperbolic
  3-manifolds and connections with large scale geometry}, Advances in
  Mathematics \textbf{303} (2016), 192--278.

\bibitem[FM01]{fenley2001quasigeodesic}
S{\'e}rgio~R. Fenley and Lee Mosher, \emph{Quasigeodesic flows in hyperbolic
  3-manifolds}, Topology \textbf{40} (2001), no.~3, 503--537.

\bibitem[Fra13]{Frankel_thesis}
Steven Frankel, \emph{Quasigeodesic flows from infinity}, Ph.D. thesis,
  University of Cambridge, 2013.

\bibitem[Fra18]{frankel_closedorbits}
\bysame, \emph{Coarse hyperbolicity and closed orbits for quasigeodesic flows},
  Annals of Mathematics \textbf{188} (2018), no.~1, 1--48.

\bibitem[FSS19]{schleimer2019veering}
Steven Frankel, Saul Schleimer, and Henry Segerman, \emph{From veering
  triangulations to link spaces and back again}, arXiv preprint
  arXiv:1911.00006 (2019).

\bibitem[Hua24]{Huang_UC}
Junzhi Huang, \emph{Depth-one foliations, pseudo-anosov flows and universal
  circles}, ArXiv e-print 2410.07559 (2024).

\bibitem[LMT23]{landry2023endperiodic}
Michael~P. Landry, Yair~N. Minsky, and Samuel~J. Taylor, \emph{Endperiodic maps
  via pseudo-{A}nosov flows}, arXiv preprint arXiv:2304.10620 (2023).

\bibitem[LMT24]{LMTstrongtst}
\bysame, \emph{Transverse surfaces and pseudo-{A}nosov flows}, ArXiv e-print
  2406.17717 (2024).

\bibitem[LT24]{LandryTsangStep1}
Michael~P. Landry and Chi~Cheuk Tsang, \emph{Endperiodic maps, splitting
  sequences, and branched surfaces}, To appear in Geometry and Topology (2024),
  1--144.

\bibitem[Mos92]{Mos92}
Lee Mosher, \emph{Dynamical systems and the homology norm of a $3$-manifold
  {I}: efficient intersection of surfaces and flows}, Duke Math. J. \textbf{65}
  (1992), no.~3, 449--500.

\bibitem[Mos96]{Mos96}
\bysame, \emph{Laminations and flows transverse to finite depth foliations},
  Unpublished monograph, 1996.

\bibitem[Nov65]{Nov65}
S.~P. Novikov, \emph{The topology of foliations}, Trudy Moskov. Mat.
  Ob\v{s}\v{c}. \textbf{14} (1965), 248--278.

\bibitem[Sha01]{shalen2001three}
Peter~B Shalen, \emph{Three-manifolds and baumslag--solitar groups}, Topology
  and its Applications \textbf{110} (2001), no.~1, 113--118.

\bibitem[Thu97]{thurston1997three}
William~P Thurston, \emph{Three-manifolds, foliations and circles, i}, arXiv
  preprint math/9712268 (1997).

\bibitem[Thu98]{thurston_circles2}
William~P. Thurston, \emph{Three-manifolds, foliations and circles {II}: the
  transverse asymptotic geometry of foliations}, Unfinished preprint, January
  1998.

\end{thebibliography}
\bibliographystyle{amsalpha}
\end{document}